\documentclass{amsart}
\usepackage{amsmath}
\usepackage{bbm}
\usepackage{latexsym}
\usepackage{mathrsfs}

\usepackage{hyperref}
\hypersetup{
    colorlinks,
    linkcolor={red!50!black},
    citecolor={blue!50!black},
    urlcolor={blue!80!black}
}

\usepackage[usenames,dvipsnames]{pstricks}
\usepackage{epsfig}
\usepackage{pst-grad} 
\usepackage{pst-plot} 
\usepackage{hyperref}
\usepackage{verbatim} 
\usepackage{mathtools} 
\usepackage[mathcal]{euscript}

\newcommand{\R}{\mathbb{R}}
\newcommand{\N}{\mathbb{N}}

\newcommand{\wh}[1]{\widehat{#1}}

\newcommand{\mc}[1]{\mathcal{#1}}

\newcommand{\Att}{A_{x_0}^T}

\newtheorem{thm}{Theorem}
\newtheorem{lemma}[thm]{Lemma}

\newtheorem{prop}[thm]{Proposition}

\theoremstyle{definition}
\newtheorem{defi}[thm]{Definition}

\theoremstyle{definition}

\theoremstyle{remark} 
\newtheorem{remark}{Remark}

\newcommand{\be}{\begin{equation}}
\newcommand{\ee}{\end{equation}}

\mathtoolsset{showonlyrefs,showmanualtags} 
\numberwithin{equation}{section}

\author{Davide Barilari}
\address{Universit\'e Paris Diderot - Paris 7, Institut de Mathematique de Jussieu, UMR CNRS 7586 - UFR de Math\'ematiques.}
\email{\href{mailto:davide.barilari@imj-prg.fr}{\nolinkurl{davide.barilari@imj-prg.fr}}}

\author{Francesco Boarotto}
\address{Sorbonne Universit\'es, UPMC Univ Paris 06, CNRS UMR 7598, Laboratoire Jacques-Louis Lions,
F-75005, Paris, France and INRIA - Team CaGe}
\email{\href{mailto:boarotto@ljll.math.upmc.fr}{\nolinkurl{boarotto@ljll.math.upmc.fr}}}

\date{\today}
\title[Kolmogorov operators in dimension two]{Kolmogorov-Fokker-Planck operators in dimension two: heat kernel and curvature}
\usepackage[numbers]{natbib}

\begin{document} 

\begin{abstract}
	We consider the heat equation associated with a class of hypoelliptic operators of Kolmogorov-Fokker-Planck type in dimension two. 
	We explicitly compute the first meaningful coefficient of the small time asymptotic expansion of the heat kernel on the diagonal, and we interpret it in terms of curvature-like invariants of the optimal control problem associated with the diffusion.  This gives a first example of geometric interpretation of the small-time heat kernel asymptotics of non-homogeneous H\"ormander operators which are not associated with a sub-Riemannian structure, i.e., whose second-order part does not satisfy the H\"ormander condition. 

\end{abstract}
\maketitle

\setcounter{tocdepth}{1}
	\tableofcontents

	\newcommand{\contr}{\Omega^{T}_{x_{0}}}
\newcommand{\K}{\mathcal{K}}

\section{Introduction}\label{sec:Intro}	
During the last years, many results  revealing the interaction between analysis and geometry have been obtained in the study of the heat equation. In particular, many of them relate the small-time asymptotics of the fundamental solution of the heat equation, the so-called \emph{heat kernel}, to geometric quantities underlying the partial differential operator, in several smooth and non-smooth contexts.

These investigations are inspired by the nowadays classical results in Riemannian geometry. In the Riemannian context there is a precise understanding of the small-time asymptotics of the heat kernel and its relation to distance, its singularities (such as cut and conjugate loci) and curvature invariants. Some fundamental results in these directions can be found in \cite{varadhan,molchanov,neelstroock,spectrebook}. For a more comprehensive account of the literature we refer to the books  \cite{bergerpan, grigbook, rosenberg} and references therein. 

We state here the following result, that can be found for instance in \cite[Chapter IV]{bismut}, and which is the closest  to the spirit of the investigation presented in this paper.
\begin{thm} \label{t:bismut}
Let $(M,g)$ be a  complete smooth $n$-dimensional Riemannian manifold, and let $\Delta_{g}$ be the Laplace-Beltrami operator of $(M,g)$.
Let $p(t,x,y)$ be the heat kernel associated with $L=X_{0}+\frac12\Delta_{g}$, where $X_{0}$ is a smooth vector field.  Fix $x_{0}\in M$. Then for $t\to 0$ one has
\begin{equation}\label{eq:riemexp}
p(t,x_{0},x_{0})=\frac{1}{(2\pi t)^{n/2}}\left(1+t\left(-\frac{1}{2}\mathrm{div}_{\mu}(X_{0})(x_{0})-\frac{1}{2}|X_{0}(x_{0})|^{2}+\frac{1}{12}S(x_{0})\right)+O(t^{2})\right),
\end{equation}
where $\mu$ is the Riemannian volume, and $S$ is the scalar curvature of the Riemannian metric $g$.
\end{thm}
For later purposes, let us mention that, in terms of an orthonormal basis $X_{1},\ldots,X_{n}$ for the Riemannian metric, $L$ rewrites as follows
\begin{equation} \label{eq:laplaciano}
	L=X_{0}-\frac12 \sum_{i=1}^nX_{i}^*X_i=X_{0}+\frac12 \left(\sum_{i=1}^{n}X_{i}^{2}+(\mathrm{div}_{\mu}X_{i})X_{i}\right),
\end{equation}
where $\mu$ is the Riemannian volume form. Recall that the divergence $\mathrm{div}_{\mu}X$ of a vector field $X$ with respect to the measure defined by the volume form $\mu$ is the function characterized by the identity
\begin{equation}
\int_{M}Xf \, d\mu = -\int_{M}f \mathrm{div}_{\mu}X \,d\mu, \qquad f\in C^{\infty}_{c}(M)
\end{equation}
for every smooth function $f$ having compact support in $M$. In \eqref{eq:laplaciano} $X^{*}$ denotes the formal adjoint of a smooth vector field $X$ on $M$ (as a differential operator). The adjoint is computed with respect to the $L^{2}$ inner product, i.e., considering $C^{\infty}_{c}(M)$ as a subspace of $L^{2}(M,\mu)$ and satisfies $X^{*}=-X-\mathrm{div}_{\mu}X$.

The extension of results in the spirit of Theorem \ref{t:bismut} to non-Riemannian situations, such as sub-Riemannian geometry or more in general hypoelliptic operators, when possible, is non-trivial: some results have been obtained relating the hypoelliptic heat kernel of a sub-Riemannian Laplacian with its associated Carnot-Carath\'eodory distance \cite{leandremaj,leandremin,hypoelliptic} and its cut locus \cite{BBCN13,BBN12}, but much less is known concerning the relation with curvature and other geometric invariants. In the 3D contact sub-Riemannian case a result in the spirit of Theorem~\ref{t:bismut} is contained in \cite{mioheat}, where an invariant $\kappa$ plays the role of the scalar curvature. Only partial results are known in higher-dimension, even for the contact case: for contact structures with symmetries in \cite{BGS84,ST84} the first coefficient in the small-time heat kernel asymptotics has been related to the scalar Tanaka-Webster curvature. See also \cite{baudoin2013subelliptic,baudoin2014subelliptic,W16} and \cite{CDVHT1} for a recent account on heat kernel asymptotics on sub-Riemannian higher-dimensional model spaces and spectral invariants of sub-Laplacians in the 3D contact case, respectively.

The difficulty discussed above is related on the one hand to the difficulty of defining a general notion of (scalar) curvature associated with a given Carnot-Carath\'eodory metric \cite{MemAMS,BR16,BR17,BR-AM}, and on the other hand (at least using the perturbation technique exploited in the present paper) to the fact that the precise knowledge of the heat kernel of the principal part of the operator is non-trivial as soon as the operator is not elliptic. The main idea of the paper is to exploit the knowledge of the principal part of the operator when it is an H\"ormander operator of Kolmogorov type. A characterization of the full small-time heat kernel asymptotics in terms of curvature-like invariants for this kind of operators have been considered in \cite{BP15}.

\medskip
The goal of this paper is to provide an analogue of Theorem \ref{t:bismut}  for the heat equation associated with a class of hypoelliptic operators of Kolmogorov-Fokker-Planck type in dimension two. To our best knowledge this is the first example of this kind of results for non-homogeneous H\"ormander operators which are not associated with a sub-Riemannian structure, i.e., whose second-order part does not satisfy the H\"ormander condition. 

We explicitly compute the first meaningful coefficient of the small-time asymptotic expansion of the heat kernel on the diagonal, and we study its relation with curvature-like invariants of the optimal control problem associated with the diffusion operator.

\subsection{The Kolmogorov operator}
Consider on $[0,+\infty)\times \R^2$ the second-order partial differential operator $\K$ defined as follows:
\be\label{eq:K}\K=\partial_{t}-x_{1}\partial_{x_{2}}-\frac12\partial_{x_{1}}^{2},\qquad (t,x_{1},x_{2})\in [0,+\infty)\times \R^2 \ee
In the literature, $\K$ is often called Kolmogorov operator, since it was first introduced in Kolmogorov's paper \cite{Kolmogoroff1931} in the study of Brownian motion and kinetic theory of gases.

In fact, Kolmogorov wrote in \cite{Kolmogorov34} an explicit smooth (which always means $C^{\infty}$ in what follows) fundamental solution  for the operator $\K$ with respect to the Lebesgue measure, thus proving its hypoellipticity.
Introducing the matrices
\begin{equation}
A=\begin{pmatrix}
0&0\\
1&0
\end{pmatrix},\quad \mbox{and}\quad
G_t=\begin{pmatrix}
t&t^{2}/2\\
t^{2}/2&t^{3}/3
\end{pmatrix},
\label{eq:Dt3}
\end{equation}
the fundamental solution $p_0:\R^+\times (\R^2)^2\to \R$ associated with \eqref{eq:K} can be explicitly rewritten as follows 
$$p_{0}(t,x,y)=\frac{e^{-\frac{1}{2}(y-e^{tA}x)^*G_t^{-1}(y-e^{tA}x)}}{2\pi 
	\sqrt{\det{G_t}}}, \qquad (t,x,y)\in [0,+\infty)\times (\R^2)^{2},$$
where $v^{*}$ denotes the transpose of a vector $v$. Notice in particular that choosing $x=y=0$ one has the on-the-diagonal  identity 
\begin{equation}\label{eq:koldiag}
p_0(t,0,0)= \frac{\sqrt{12}}{2\pi t^{2}}.
\end{equation}

The Kolmogorov operator $\K$ is one of the simplest examples of a homogeneous, left-invariant and hypoelliptic operator that we can define on an homogeneous group structure  $\mathbb{G}$ on $\R^3$. To see this, let us endow $\mathbb{G}=\R^3$ with the following translations and dilations:
\begin{align}
	(x_1,x_2,t)\circ (y_1,y_2,s)&:=(x_1+y_1,x_2+y_2-x_1s,t+s),\\
	\delta_\lambda(x_1,x_2,t)&:=(\lambda x_1,\lambda^3x_2,\lambda^2t),\quad \lambda\in \R^+.
\end{align}
Then one can see  $\mathbb{G}$ as a homogeneous and stratified Lie group (in the sense of \cite{lanconellibook}) where  $X_0=x_1\partial_{x_2}$,  $X_1=\partial_{x_1}$ and $\partial_t$ are left-invariant vector fields on $\mathbb{G}$. Moreover, notice that $[X_{0},X_{1}]=-\partial_{x_{2}}$. In particular \begin{equation}\label{eq:hc}
	\textrm{span}_x\{X_1,[X_0,X_1]\}=T_xM,\qquad \forall\,x\in M,
\end{equation} i.e., the vector fields $X_0,X_1$ satisfy the H\"ormander condition $\mathrm{Lie}_x\{ (\mathrm{ad} X_0)^jX_1\mid j\geq 0 \}=T_xM$, for every $x\in M$. Thanks to the celebrated result of \cite{hormander},
\begin{equation}\label{eq:kolmogorot}
\K=\partial_t-\left(X_0+\frac 12 X_1^2\right)
\end{equation}
is indeed an hypoelliptic partial differential operator on $\mathbb{G}$, 
which is left-invariant with respect to the group product and homogeneous of degree two with respect to the family of dilations.


The study of such operators has recently known a quick development. For the interested reader we refer for example to \cite{LancoPoli,Poli94,bramantibook,GarofaloEMS,BaudoinEMS,Vol1EMS} and references therein.
 
\subsection{Perturbing the Kolmogorov operator}

We consider in this paper a class of hypoelliptic H\"ormander operators that can be modeled on $\K$.  In what follows $M$ denotes $\R^2$ or  a smooth, connected and compact two-dimensional manifold, endowed with a volume form $\mu$ (or a smooth density), we study the class of second-order differential operators of Fokker-Planck type on $\R\times M$ written in the form $\partial_{t}-L$ where 
\begin{equation} \label{eq:mainoper}
L=X_0-\frac12 X_1^*X_1=X_{0}+\frac12  \left(X_{1}^{2}+\mathrm{div}_\mu(X_1)X_1\right),
\end{equation}
and the vector fields $X_0,X_1$ satisfy condition \eqref{eq:hc}. This class of operators does indeed contain \eqref{eq:K}, as explained above, cf.~\eqref{eq:kolmogorot}. Notice that the vector field $X_1=\partial_{x_{1}}$ has divergence zero with respect $\mu$, choosing $\mu$ as the Lebesgue measure in $M=\R^2$. 
\begin{remark}
The presence of the divergence term in \eqref{eq:mainoper} does not really affect the class of operators we are considering since the first-order term can be always included in the drift by defining $Y_{0}:=X_{0}+\frac12 \mathrm{div}_\mu(X_1)X_1$ and writing 
\begin{equation} \label{eq:mainoper2}
L=Y_{0}+\frac12 X_{1}^{2}.
\end{equation}
Moreover  the pair $X_{0},X_{1}$ satisfies \eqref{eq:hc} if and only if $Y_{0},X_{1}$ does. 
\end{remark}
Writing $L$ in the form \eqref{eq:mainoper} is mainly useful for the final geometric interpretation of the small-time heat kernel expansion. Even in the case of a second-order elliptic operator, for the geometric interpretation of the heat kernel it is useful to write it as a sum of a self-adjoint part (i.e., a Laplace-Beltrami operator for some suitable Riemannian metric) plus a drift, as Theorem \ref{t:bismut} and formula \eqref{eq:laplaciano} suggest.

The H\"ormander condition guarantees the existence of a (local) smooth fundamental solution $p\in C^\infty(\R^+\times M\times M)$ for $\partial_t-L$  with respect to the measure $\mu$, and our objective is to retrieve geometrical informations on $M$ from the small-time asymptotics of $p$ on the diagonal $\Delta=\{(x,x)\mid x\in M\}\subset M\times M$. 

More precisely, given any reference point $x_0\in M$, we want to understand the expansion $p(t,x_0,x_0)$ for $t\to 0$, and see how it reflects the local geometry of the manifold around $x_0$. It is well-known from the works of Ben Arous and L\'eandre \cite{benarousdiag,benarousleandrediag} that in order to expect a polynomial decay for the heat kernel on the diagonal (as in the elliptic case, cf.\ \eqref{eq:riemexp}), we have to make the additional assumption that $X_0(x_0)$ is parallel to $X_1(x_0)$, with $X_0(x_0)$ possibly zero. 

When the two vectors $X_1(x_0)$ and $X_0(x_0)$ are independent, one may expect an exponential decay for the heat on $\Delta$. An explicit example of this behavior is showed in \cite{BP15} for the class of higher-dimensional Kolmogorov operators.

The relation between the anisotropic diffusion generated by $\partial_t-L$ with its underlying geometry is primarily given by the probabilistic characterization of the fundamental solution $p(t,x,y)$ as the probability density of the stochastic process $t\mapsto \xi_t$ that solves the stochastic differential equation
\begin{equation}\label{eq:sde}
	d\xi_t=X_0(\xi_t)+X_1(\xi_t)\circ dw,\quad \xi_0=x\in M.
\end{equation}
 Here $w$ denotes the standard Brownian motion on $M$ and $\circ$ indicates the integration in the Stratonovich sense. Moreover, it is known from the pioneering ideas of \cite{malliavin1978stochastic,stroock1981malliavin,azencott1980grandes} that when $t\to 0$, the diffusion tends to concentrate along the optimal trajectories of the following optimal control problem
\be\label{eq:cont}
\dot{x}=X_0(x)+uX_1(x),\quad x\in M,\quad u\in\R,
\ee 
where one minimizes the energy 
\be \label{eq:cont2}
J_t(u)=\frac12\int_0^t|u(s)|^2ds
\ee among all controls $u\in L^2([0,t],\R)$ such that the corresponding solution to  \eqref{eq:cont} joins  $x$ and $y$ in time $t$.
 In other words, for every fixed $x_0\in M$ and $t>0$, we consider the \emph{attainable set} $A_{x_0}^t:=\{x_u(t)\mid u\in \mc{U},\,x_u(0)=x_0\}$, where $\mc{U}\subset L^2([0,t],\R)$ is the open subset of $L^2([0,t],\R)$ such that the corresponding solution $x_u:[0,t]\to M$ to \eqref{eq:cont} is defined on the interval $[0,t]$. We introduce also the \emph{value function} $S_{x_0}^t:M\to \R\cup \{+\infty\}$, by 
\[ S_{x_0}^t(y):=\inf\{J_t(u)\mid u\in \mc{U},\,x_u(0)=x_0,\,x_u(t)=y\},\] with the convention that $\inf \emptyset=+\infty$.  For $y\in A_{x_0}^t$, solving the optimal control problem means then finding the controls $u$ realizing the infimum in $S_{x_0}^t(y)$.

Under the assumption \eqref{eq:hc}, the attainable set has non-empty interior, and the value function is always smooth along optimal trajectories for small time, see e.g.\ \cite[Appendix~A]{MemAMS}. If moreover $X_0(x_0)$ is parallel to $X_1(x_0)$, it follows from \cite{bianchinistefani,Paoli17} that for small times $t$ we can relate the attainable set $A_{x_0}^t$ to the attainable set $A_{x_0}^{t,\mc{K}}$, reached by solutions to \eqref{eq:cont} where $X_0=x_1\partial_{x_2}$ and $X_1=\partial_{x_1}$ are the vector fields defining the Kolmogorov operator $\mc{K}$. In particular, since $A_{x_0}^{t,\mc{K}}$ contains a full neighborhood of $x_0$, the same holds for $A_{x_0}^t$ and the value function $S_{x_0}^t$ is well-defined in a neighborhood of $x_0$.

Curvature-like invariants of the control system \eqref{eq:cont} are obtained by higher-order derivatives of the value function along optimal trajectories, and as such really describe the local geometry of $M$ around $x_0$. For more details to curvature-like invariants in this setting we refer the reader to Section \ref{sec:Curv}, and for a more comprehensive presentation to  \cite{MemAMS}.

We stress that these curvature-like invariants are indeed a generalization of the Riemannian curvature. More precisely, when one applies this approach to a control systems associated with a Riemannian geodesic problem, the value function is related to the squared Riemannian distance and the curvature-like invariants recover the full Riemann curvature tensor \cite[Chapter 4]{MemAMS}.

\subsection{Main result and comments}

To have a clear geometric interpretation of the small-time asymptotics, it is natural to choose for the operator $L$ defined as in \eqref{eq:mainoper}, a volume form $\mu$ that represents some canonical volume form associated with the problem, as in the Riemannian case where one chooses the canonical Riemannian volume.

In view of the assumption \eqref{eq:hc}, it is natural to consider on the two-dimensional manifold $M$ the volume form $\mu$ 
defined by the requirement that 
\be\label{eq:w}
\mu(X_1,[X_0,X_1])=1.
\ee 
Inspired by Theorem \ref{t:bismut}, we expect the small-time asymptotics of $p(t,x_0,x_0)$ to contain informations on the curvature-like invariants associated with the system \eqref{eq:cont}, and to depend on the pointwise values both of $X_0(x_0)$ and $\mathrm{div}_\mu(X_0)(x_0)$. The main result of this paper is the following.

\begin{thm}\label{t:main}
	Let $M$ be a two-dimensional smooth, connected and compact manifold, or $M=\R^{2}$. Fix $x_0\in M$ and let $X_0$ and $X_1$ be two smooth vector fields on $M$, with bounded derivatives of all orders, such that 
	\begin{itemize}
	\item[(a)] $X_0(x_0)$ is parallel to $X_1(x_0)$, 
	\item[(b)]
	
$\mathrm{span}_{x}\{X_1,[X_0,X_1]\}=T_xM$,   for every $x\in M$.

\end{itemize}
  Let $\mu$ be the volume form defined in \eqref{eq:w} and $p(t,x,y)$ be the heat kernel associated with 
\begin{equation} \label{eq:operatoreboa}
\partial_t-X_0-\frac{1}{2}\left(X_1^2+(\mathrm{div}_{\mu}X_{1})X_{1}\right).
\end{equation}
Then the following asymptotic expansion holds for $t\to 0$
	\begin{equation} \label{eq:lunga}
	p(t,x_0,x_0)=\frac{\sqrt{12}}{2\pi t^2}\left(1+t\left(-\frac12 \mathrm{div}_\mu(X_0)(x_0)-\frac12 |X_0(x_0)|^2+\frac1{14}K_1(x_0)-\frac{1}{70}K_2(x_0)^2\right)+o(t) \right).	
	\end{equation}	
\end{thm}
The main-order coefficient in \eqref{eq:lunga} is the one of Kolmogorov operator \eqref{eq:koldiag}, while the contribution of the drift is analogous to what happens in \eqref{eq:riemexp}. The two remaining coefficients are curvature-like invariants associated with the control problem and play the role of the scalar curvature. All coefficients appearing in \eqref{eq:lunga} have invariant meaning.

Few comments on the statement of Theorem \ref{t:main} are in order.

(i).~The manifold $M$ is not endowed by a canonical Riemannian metric, the symbol $|X_0(x_0)|$ here denotes the norm of the vector $X_0(x_0)$ in the vector space $D_{x_{0}}:=\R X_{1}(x_{0})$ endowed with the scalar product defined by declaring $X_{1}(x_{0})$ as an orthonormal basis.

(ii).~The assumption for the vector fields to have bounded derivatives of all orders is only technical (and automatically satisfied when $M$ is compact). Indeed one can weaken the condition above to a ``polynomial growth condition'' on the vector fields to prove H\"ormander Theorem, cf.~\cite[Assumption 4.2 and Theorem 4.5]{hairer}. This plays the role of a completeness assumption (cf.~Theorem \ref{t:bismut}) and, as soon as the operator \eqref{eq:operatoreboa} admits a fundamental solution,  thanks to the localization arguments contained in Section~\ref{sec:PertKol} one reduces for the small-time asymptotics on the diagonal to the assumptions of Theorem \ref{t:main}.

(iii).~The quantities $K_1$ and $K_2$ are curvature-like invariants associated with the control problem \eqref{eq:cont}-\eqref{eq:cont2}, defined in Section \ref{sec:Curv}. In particular $K_1$ can be interpreted as the contraction along the controlled vector field $X_1$ of the second-order part of the curvature operator associated with the control problem, while $K_2$ plays the role of a ``higher order'' curvature. The presence of both terms is related with the fact that the drift $X_{0}$ is necessary to fulfill the H\"ormander condition.

For a detailed and general presentation of curvature-like invariants associated with an affine optimal control problem, we refer to the monograph \cite{MemAMS}. In Section \ref{sec:Curv} we give an adapted presentation to the specific context of the paper. We define and compute the constants $K_1$ and $K_2$ in terms of the structure constants of the Lie algebra defined by $X_{0}$ and $X_{1}$.

\subsection{On the role of the volume}
 Let us denote by $L^{\mu}$ the operator defined in \eqref{eq:mainoper} with respect to the volume $\mu$. If one is interested in the study of the operator $L^{\omega}$ associated with a different smooth volume $\omega$, which is then  absolutely continuous with respect to $\mu$, it is not difficult to understand that this is actually equivalent to the study of the same operator with a new drift vector field. More precisely, writing $\omega=e^{\psi}\mu$ for some smooth function $\psi\in C^{\infty}(M)$, the formula $$\mathrm{div}_\omega(X_1)=\mathrm{div}_\mu(X_1)+(X_{1}\psi)$$ shows that the study of the operator $L^{\omega}$ reduces to the study of the original one (i.e., where the second-order term is symmetric with respect to $\mu$) but with modified drift $X'_{0}=X_{0}+\frac12(X_{1}\psi)X_{1}$. See also \cite{ABP16} for a discussion about the role of the curvature in the variation of the volume along a geodesic flow. 

\subsection{Structure of the paper} In Section~\ref{sec:FundSol} we introduce some preliminaries, we present the perturbation formula of Duhamel and we explain how the blow-up procedure applied to a hypoelliptic operator $\partial_t-L$ permits to recover its nilpotent part.  
Section~\ref{sec:PertKol} constitutes the computational core of this work. We specialize the theory previously exposed to the specific case of a two-dimensional manifold and a Kolmogorov-like operator. We show how a fundamental solution of $\partial_t-L$ can be reconstructed from a ``weighted'' parametrix method applied to the fundamental solution of $\K$. Despite the low dimensionality of the problem the calculations involved are heavy, nonetheless everything can be done explicitly.  

Section~\ref{sec:Curv} is devoted to a brief account on curvature-like invariants associated with an affine control system associated to our operator, and we compute them explicitly. 
Finally, in Section~\ref{sec:Coordinates}, we collect together the results to interpret the first coefficient appearing in the small-time asymptotics of the fundamental solution, proving Theorem~\ref{t:main}.

\section{The fundamental solution and its perturbation}
\label{sec:FundSol}
In this section  $M$ denotes either $\R^n$, or a smooth, connected, compact manifold of dimension $n$. Let also $\mu$ be a given smooth volume form on $M$.

Given any collection $X_0,X_1,\dots,X_k$, of smooth (i.e., $C^\infty$) vector fields on $M$  with bounded derivatives of all orders and that satisfy the H\"ormander condition
\begin{equation}
	\label{eq:Hor}
	\textrm{Lie}_x\left\{\left(\mathrm{ad}\,X_0\right)^jX_i\mid j\geq 0,\, i=1,\dotso,k\right\}=T_xM \quad\text{for every }x\in M,
\end{equation}
we consider the differential operator $\partial_t-L$, defined on the space of distributions $\mc{D}'(\R^+\times M)$ as
\begin{equation}
	\label{eq:Operator}
	\left(\partial_t-L\right)(\varphi)
	=\partial_t(\varphi)-X_0(\varphi)
	-\frac{1}{2}\sum_{i=1}^kX_i^2(\varphi),\quad \varphi\in \mc{D}'(\R^+\times M).
\end{equation}
Notice that the operator $L$ is smooth in the domain under consideration.
\begin{defi}[Fundamental solution]\label{def:fundsol}
	We call a \emph{fundamental solution} to the operator $\partial_t-L$ with respect to $\mu$  a 
function  $p(t,x,y)\in C^{\infty}(\R^+\times M\times M)$ such that:
	\begin{itemize}
		\item [(a)] $(\partial_t-L)p(t,x,y)=0$ (where $L$ acts on the $x$-variable) for every fixed $y\in M$,
		\item [(b)] $\lim_{t\to 0}p(t,x,y)=\delta_y(x)$ in the following sense:
for every $\varphi\in C^\infty_{c}(M)$ there holds
\be
	\lim_{t\to 0}\int_Mp(t,x,y)\varphi(y)d\mu(y)=\varphi(x),\quad\forall\,x\in M,\quad\text{uniformly in }t.
\ee

	\end{itemize}
\end{defi}

A crucial result of H\"ormander \cite{hormander} states that, under condition \eqref{eq:Hor}, the operator $\partial_t-L$ in \eqref{eq:Operator} is hypoelliptic, that is for every distribution $v$ defined on a domain $\Omega\subset \R\times M$, the condition $(\partial_t-L)(v)\in C^\infty(\Omega)$ implies $v\in C^\infty(\Omega)$. 
In particular, it admits a fundamental solution $p$ (that is smooth) with respect to the given volume $\mu$,  given explicitly as the probability density of the process $t\mapsto \xi_t$ that solves the following stochastic differential equation (in the Stratonovich sense) \be\label{eq:stratonovich}\left\{\begin{aligned}d\xi_t&=X_0(\xi_t)+\sum_{i=1}^k X_i(\xi_t)\circ db_i,\\ \xi_0&=x,\end{aligned}\right.\ee where $b=(b_1,\dots,b_k)$ is a $k$-dimensional Brownian motion, see for instance \cite{hairer}. 

\begin{defi}[Weighted coordinates]
	Let $(U,x)$ be a coordinated neighborhood of $x_0$, and let $w=(w_1,\dotso, w_n)$ be an $n$-tuple of positive integers. For any $0<\varepsilon\leq 1$ we define the $\varepsilon$-dilation $\delta_{\varepsilon}$ as follows:
	\be 
		\delta_{\varepsilon}(x_1,\dotso,x_n)=(\varepsilon^{w_1}x_1,\dotso,\varepsilon^{w_n}x_n),\qquad\forall\,x=(x_1,\dotso,x_n)\in U.
	\ee
\end{defi}
The action of $\delta_{\varepsilon}$ extends to vector fields and differential forms through the differential $\delta_{\varepsilon*}$ and its adjoint $\delta^{*}_{\varepsilon}$, respectively. Hence one can compute the action of the dilation on coordinate vector fields $\partial_{x_{i}}$ and the volume form $\mu$ as follows
\be 
	(\delta_{\varepsilon*})\partial_{x_i}=\varepsilon^{-w_i}\partial_{x_i},\qquad (\delta_{\varepsilon})^*\mu=\varepsilon^{\sum_{i=1}^kw_i}\mu.
\ee
If we perform a rescaling of time and space around $x_0$, the fundamental solution $p$ changes according to the following proposition \cite[Proposition 2.4]{Paoli17}.
\begin{prop}
	\label{prop:rescaling}
	Fix weights $(w_1,\dotso,w_n)$, $0<\varepsilon\leq 1$ and $\gamma\in \N$. Then a fundamental solution $q_\varepsilon$ to the rescaled operator
	\be\label{eq:rescop}
		\partial_t-\varepsilon^\gamma \left(\delta_{1/\varepsilon*}X_0+\frac{1}{2}\sum_{i=1}^k(\delta_{1/\varepsilon*}X_i)^2\right)
	\ee
	is given by 
	\begin{equation}\label{eq:nnnn} 
		q_\varepsilon(t,x,y)=\varepsilon^{\sum_{i=1}^kw_i}p(\varepsilon^\gamma t,\delta_\varepsilon x,\delta_\varepsilon y),
	\end{equation}
	where $p(t,x,y)$ is a fundamental solution associated with the operator \eqref{eq:Operator}.
\end{prop}

When we apply the push-forward of a dilation $\delta_{\varepsilon}$ to the vector field $X_i$, depending on the values of the weights $w=(w_1,\dotso, w_n)$ and on its Taylor expansion, there exist 
	integers $\alpha_i,\beta_i$ with $\alpha_i>\beta_i$, 
	a vector field $\widehat{X}_i$, homogeneous of degree $-\alpha_i$ with respect to the dilation, and a vector field $\wh{Y}_i^\varepsilon$ which is smooth with respect to $\varepsilon$, such that
\be
	\label{eq:rescVF}
	\delta_{1/\varepsilon*}X_i=\frac{1}{\varepsilon^{\alpha_i}}\widehat{X}_i+
	\frac{1}{\varepsilon^{\beta_i}}\wh{Y}^\varepsilon_i.
\ee
In the sequel, we will short-handedly refer to \eqref{eq:rescVF} by saying that 
\[
	\delta_{1/\varepsilon*}X_i=\frac{1}{\varepsilon^{\alpha_i}}\widehat{X}_i+o(\varepsilon^{-\alpha_i}).
\]
In particular, we can rewrite \eqref{eq:rescop} as
\be
	\label{eq:prinPart}
	\partial_t-\varepsilon^\gamma \left(\frac{1}{\varepsilon^{\alpha_0}}\widehat{X}_0+\frac{1}{2}\sum_{i=1}^k\frac{1}{\varepsilon^{2\alpha_i}}(\widehat{X}_i)^2\right)+o(\varepsilon^{\gamma-\alpha}),\quad \alpha=\max\{\alpha_0,2\alpha_1,\dotso,2\alpha_k\}.
\ee
If we can find suitable coordinates and a good choice of the weights, so that all the vector fields $X_0$ and $X_i$ rescale with the same degree $\gamma$, then for every $0< \varepsilon\leq 1$ the principal part of 
\be\label{eq:L'}
L_\varepsilon:=\varepsilon^\gamma \left(\delta_{1/\varepsilon*}X_0+\frac{1}{2}\sum_{i=1}^k(\delta_{1/\varepsilon*}X_i)^2\right)
\ee reduces to
\be
	\label{eq:L0}
	L_0=\widehat{X}_0+\frac{1}{2}\sum_{i=1}^k\widehat{X}_i^2,
\ee
and we are able to write a decomposition of $L_\varepsilon$ into a principal part $L_0$ plus a remainder term $\mc{X}_\varepsilon$ going to zero for $\varepsilon\to 0$.

\subsection{Duhamel's perturbation formula}

In this section we give only a brief overview of the method to keep the paper self-contained. The interested reader is referred to \cite{Fritz,rosenberg,mioheat} for a more comprehensive exposition.

Let $\mc{L}$ be an operator that acts on $L^2(M,\mu)$ and admits a fundamental solution $p$. Then we may consider the evolution operator $e^{\mc{L}}:\R^+\times L^2(M,\mu)\to L^2(M,\mu)$, defined by
\be
	e^{t\mc{L}}\varphi(x):=e^{\mc{L}}(t,\varphi)(x)=\int_Mp(t,x,y)\varphi(y)d\mu(y).
\ee
Moreover, in this context, we refer to $p$ as an \emph{heat kernel} (or a \emph{Schwartz kernel}) for the operator $e^{\mc{L}}$. It is then easy to show that, for every $\varphi\in L^2(M,\mu)$, there hold both
\be
	\partial_t e^{t\mc{L}}\varphi=\mc{L}e^{t\mc{L}}\varphi\quad\textrm{and}\quad \lim_{t\to 0}e^{t\mc{L}}\varphi=\varphi.
\ee
In particular $e^{\mc{L}}$ is a \emph{heat operator}. Consider now the case where $\mc{L}$ admits a decomposition
\be
	\mc{L}=\mc{L}_0+\mc{X},
\ee
into a principal part plus a perturbation. If also $\mc{L}_0$ admits a fundamental solution $p_0$, we may reconstruct $e^{t\mc{L}}$ from $e^{t\mc{L}_0}$ as follows.
\begin{thm}[Duhamel]\label{thm:Duhamel}
	For every $t\geq 0$ the following identity holds:
	\be
		e^{t\mc{L}}=e^{t\mc{L}_0}+\int_0^te^{(t-s)\mc{L}}\mc{X}e^{s\mc{L}_0}ds=e^{t\mc{L}_0}+e^{t\mc{L}}*\mc{X}e^{t\mc{L}_0}.
	\ee
	Here $*$ denotes the convolution of two operators $A(t,\cdot)$ and $B(t,\cdot)$, acting on the Hilbert space $L^2(M,\mu)$, and reads as follows:
	\be
		(A*B)(t,\varphi)=\left(\int_0^tA(t-s)B(s)ds\right)(\varphi):L^2(M,\mu)\to L^2(M,\mu).
	\ee
	Moreover, assume that both $A(t,\cdot)$ and $B(t,\cdot)$ possess a heat kernel, denoted by $a(t,x,y)$ and $b(t,x,y)$ respectively. Then a heat kernel of $(A*\mc{X}B)(t)$ can be computed as:
	\be
		(a*\mc{X}b)(t,x,y):=\int_0^t\int_Ma(s,x,z)\mc{X}_zb(t-s,z,y)d\mu(z)ds,
	\ee
	where the operator $\mc{X}$ acts on the $z$-variable within the integral.
\end{thm}
Choosing $A(t)=e^{t\mc{L}}$ and $B(t)=e^{t\mc{L}_0}$,  Theorem~\ref{thm:Duhamel} permits to write
\be
	p(t,x,y)=p_0(t,x,y)+(p*\mc{X}p_0)(t,x,y)
\ee or, more precisely
\be
	p(t,x,y)=p_0(t,x,y)+\sum_{i=1}^r(p_0*^{i}\mc{X}p_0)(t,x,y)+(p*^{r+1}\mc{X}p_0)(t,x,y),
\ee
where $*^{i}\mc{X}p_0$ indicates the $i$-th repeated convolution.

\subsection{Perturbation of the fundamental solution}

We have seen how, by means  of an appropriate weighted dilation $\delta_{\varepsilon}$, for every $0<\varepsilon\leq 1$ we can express $L_\varepsilon$ in \eqref{eq:L'} as the sum $L_\varepsilon=L_0+(L_\varepsilon-L_0)=L_0+\mc{X}_\varepsilon$. If additionally $L_0$ admits a fundamental solution $q_0(t,x,y)$, then Duhamel's formula implies that
\be\label{eq:duh}
	q_\varepsilon(t,x_0,x_0)=q_0(1,x_0,x_0)+\sum_{i=1}^r(q_0*^{i}\mc{X}_\varepsilon q_0)(1,x_0,x_0)+(q_\varepsilon*^{r+1}\mc{X}_\varepsilon q_0)(1,x_0,x_0).
\ee
Since by \eqref{eq:nnnn} one has $p(t,x,y)=\varepsilon^{-\sum_{i=1}^kw_i}q_\varepsilon(\varepsilon^{-\gamma}t,\delta_{1/\varepsilon}x,\delta_{1/\varepsilon}y)$, and the dilation $\delta_{\varepsilon}$ is centered at $x_0$, we eventually find that
\begin{align}
	\label{eq:duhamel}
	p(t,x_0,x_0)&=t^{-\frac{\sum_{i=1}^kw_i}{\gamma}}q_{\sqrt[\gamma]{t}}(1,x_0,x_0)\\
				&=t^{-\frac{\sum_{i=1}^kw_i}{\gamma}}\left(q_0+\sum_{j=1}^r(q_0*^j\mc{X}_{\sqrt[\gamma]{t}}q_0)+(q_{\sqrt[\gamma]{t}}*^{r+1}\mc{X}_{\sqrt[\gamma]{t}}q_0)\right)(1,x_0,x_0).
\end{align}

\subsection{Graded coordinate chart}

Evidently, the role of the drift field $X_0$ is different from that of the $X_i$. Since the latter are applied twice as many times as the former, it is natural to assign $X_0$ the weight $2$, and to all the other vector fields the weight $1$. The length of any bracket $\Lambda\in\textrm{Lie}\{X_0,\dotso,X_k\}$ is then computed by:
\be
	|\Lambda|=2|\Lambda|_0+\sum_{i=1}^k|\Lambda|_i,\quad |0|=0,
\ee
where $|\Lambda|_j$ counts the number of occurrences of $X_j$ within $\Lambda$.

By means of these weights, we construct a filtration $\{\mc{G}_i\}_{i\in \N}\subset TM$ as follows:
\be
	\mc{G}_0=\{0\},\quad \mc{G}_i=\textrm{span}\{\Lambda\in\textrm{Lie}\{X_0,\dotso,X_k\}:|\Lambda|\leq i\}.
\ee
Notice that $\mc{G}_i\subset\mc{G}_{i+1}$, $[\mc{G}_i,\mc{G}_l]\subset\mc{G}_{i+l}$ and that eventually, by H\"ormander's condition \eqref{eq:Hor}, $\mc{G}_m$ becomes equal to $TM$, for some minimal  integer $m$ called the \emph{step} of the filtration.

For a fixed point $x_0\in M$, the collection $\{\mc{G}_i(x_0)\}_{i=1}^m$ stratifies $T_{x_0}M$. If we set $d_0=0$ and $d_i=\dim \mc{G}_i(x_0)$, then this stratification induces a particular choice of adapted coordinates around $x_0$. Concerning this construction, we report here only the results which are relevant for this paper. For an exhaustive description of the graded structure associated with such coordinates, and how it defines a grading also on the algebra of differential operators, we refer the interested reader to the original construction in \cite{bianchinistefani} (see also \cite{mioheat,Paoli17} for its relation with heat kernel).

\begin{prop}[\protect{\cite[Corollary 3.1]{bianchinistefani}}]\label{prop:coord}
	There exists a chart $(U,x)$, centered at $x_0$, such that for every $1\leq i \leq m$:
	\begin{itemize}
		\item [a)] $\mc{G}_i(x_0)=\mathrm{span}\{\partial_{x_1},\dotso,\partial_{x_{d_i}}\}$,
		\item [b)] $D_{x_h}(x_0)=0$, for every differential operator $D\in\mc{A}_i=\{Z_1\dotso Z_l,\,\mid\, Z_s\in\mc{G}_{i_s}(x_0)\textrm{ and }i_1+\dotso+i_l\leq i\}$ and every $h>d_i$.
	\end{itemize}
\end{prop}
With coordinates chosen as in the previous proposition, we define a splitting
\be
	\R^n=\R^{k_1}\oplus\dotso\oplus \R^{k_m},\quad k_i=d_{i}-d_{i-1}.
\ee
Then, the action of the dilation $\delta_{\varepsilon}:U\to U$ can be described as follows:
\be
	\delta_{\varepsilon}(x)=(\varepsilon x^1,\dotso,\varepsilon^mx^m),\quad x^i=(x_{d_{i-1}+1},\dotso, x_{d_{i}}).
\ee 
In particular, if we further assume that
\be\label{eq:X0} X_0(x_0)\in\textrm{span}\left\{X_i(x_0),[X_i,X_j](x_0)\mid i,j\in\{1,\dots,k\} \right\},\ee the choice of the weighted dilation $\delta_{\varepsilon}$ induced by Proposition~\ref{prop:coord} implies that $$\textrm{Lie}_{x_0}\left\{\left(\mathrm{ad}\,\widehat{X}_0\right)^j\widehat{X}_i\mid j\geq 0,\, i=1,\dotso,k\right\}=T_{x_0}M,$$ i.e., the Lie algebra generated by the principal parts of $X_0$ and $X_i$, for $1\leq i\leq k$, still satisfies the H\"ormander condition \eqref{eq:Hor} in a neighborhood of $x_0$. As a consequence, the operator $L_0$ in \eqref{eq:L0} is hypoelliptic, and possesses a smooth fundamental solution $p_0$ around $x_0$.

\subsection{Convergence of the fundamental solution}\label{subsec:conv} Let $$\mc{N}(x_0):=\sum_{i=1}^mi k_i=\sum_{i=1}^mi\left(\textrm{dim}(\mc{G}_i(x_0))-\textrm{dim}(\mc{G}_{i-1}(x_0))\right).$$ Within the adapted chart $(U,x)$ discussed in Proposition~\ref{prop:coord}, and with the associated choice of the weights, we have the following convergences (in the $C^\infty$ topology on $U$)\footnote{More generally, for every smooth measure $\mu$, $\lim_{\varepsilon\to 0}\frac{1}{\varepsilon^{\mc{N}(x_0)}}\delta_{\varepsilon}^*(\mu)=cd^n\mc{L}$  where $d^n\mc{L}$ denotes the $n$-dimensional Lebesgue measure, and $c$ is a real constant.} 
\be\label{eq:convergences}
	\varepsilon\delta_{1/\varepsilon*}X_i\to \widehat{X}_i,\quad \varepsilon^2\delta_{1/\varepsilon*}X_0\to \widehat{X}_0,\quad \frac{1}{\varepsilon^{\mc{N}(x_0)}}\delta_{\varepsilon}^*(\mu)\equiv \mu,
\ee
 where $\widehat{X}$ denotes the principal part of the vector field $X$ with respect to the weighted dilation $\delta_{\varepsilon}$. Assume that, by an iterated use of Duhamel's formula \eqref{eq:duh}, for any $0<\varepsilon\leq 1$ we are able to write $$q_\varepsilon=q_0+\varepsilon \mc{X}_1+\varepsilon^2\mc{X}_2+\dots\quad \textrm{on $U$.}$$
Then by the Trotter-Kato theorem \cite{Kato,trotter1958approximation} we derive the pointwise convergence for the heat evolution operators generated by $L_\varepsilon$, $0<\varepsilon\leq 1$, and $L_0$, that is
$$e^{tL_\varepsilon}\stackrel{\varepsilon\to 0}{\longrightarrow}e^{tL_0},\quad \forall\,t\in[0,+\infty).$$ This, in turn, permits to conclude that $q_\varepsilon\to q_0$ in the (weak) topology of distributions $\mc{D}'(\R^+\times U\times U)$. If we can show that $q_\varepsilon$ remains uniformly bounded in $C^\infty(\R^+\times U\times U)$ topology as $\varepsilon$ goes to zero, we can conclude that the convergence $q_\varepsilon\to q_0$ takes place also in the (strong) $C^\infty$ topology. This is a consequence of the fact that $C^\infty(\R^+\times U\times U)$ is a Montel space, and that in a Montel space every closed and bounded subset is necessarily compact.\footnote{This argument is borrowed from the ongoing work \cite{CHT2}, with the agreement of the authors.}

However, we observe that, in general, uniform estimates for $q_\varepsilon$ in the $C^\infty$ topology are usually delicate and rather hard to obtain.
	
\section{Perturbing the Kolgomorov operator}\label{sec:PertKol}
Recall that, in our setting, $M$ denotes either $\R^2$, or a smooth, connected and orientable manifold of dimension $2$. Let also $\mu$ be a given smooth volume form (or smooth density) on $M$.

The goal of this section is the study of the small-time asymptotics on the diagonal at a point $x_0\in M$ of a fundamental solution $p$ of the perturbed Kolmogorov operator
\be
	\label{eq:Kol}
	\partial_t-L=\partial_t-X_0-\frac{1}{2}\left(X_1^2+\mathrm{div}_\mu(X_1)X_1\right)
\ee
defined on $\R\times M$, with the assumptions that $X_0$, $X_1$ are smooth vector fields on $M$ with bounded derivatives of all orders, and that for every $x\in M$
\begin{equation}\label{eq:kkk}
\mathrm{dim}\left(\mathrm{span}_{x}\{X_1,[X_0,X_1]\}\right)=2.
\end{equation} 
The volume form $\mu$ is defined by the equality $\mu(X_1, [X_0,X_1])=1$. To apply the technical machinery developed in the previous section, we interpret in what follows $Y_0:=X_0+\frac{1}{2}\mathrm{div}_\mu(X_1)X_1$ as a new drift.

In order to avoid exponential behaviors  for $p(t,x_0,x_0)$, $t> 0$, and to recover meaningful geometrical informations from this expansion, we assume $X_0(x_0)\in\mathrm{span}\{X_1(x_0)\}$, see for instance \cite{benarousleandrediag,BenArous}.

Fix a neighborhood $V$ of $x_0$, on which $X_1$ and $[X_0,X_1]$ are linearly independent, and let $U\subset V$ be a coordinate chart centered at $x_0$, constructed as in Proposition~\ref{prop:coord}, compactly contained in $V$. With the aid of a localization argument on the vector fields $X_0$ and $X_1$ within $U$, we may assume that
\be\label{eq:gradedcoord}
	X_1\big|_U=\partial_{x_1},\quad X_0\big|_U=\alpha_1\partial_{x_1}+\alpha_2\partial_{x_2},\quad\alpha_1,\alpha_2\in C^\infty(U).
\ee
As a consequence of hypotheses (a) and (b) in Theorem~\ref{t:main}, we have the following crucial conditions
$$\alpha_2(0)=0, \qquad \partial_{x_1}\alpha_2(0)\neq 0.$$
Accordingly, the condition $\mu(X_1, [X_0,X_1])=1$ translates into \be\label{eq:mucoord}\mu=-\frac{1}{(\partial_{x_1}\alpha_2)}dx_1\wedge dx_2.\ee

Consider the localized version $\partial_t-L\big|_U$ of the operator in \eqref{eq:Kol}, that is, we restrict the action of $L$ to $C^{\infty}_{c}(U)$, and we consider the Friedrich extension of the corresponding operator to $L^{2}(U,\mu)$. This corresponds to fix Dirichlet boundary conditions on $U$, and to consider the action of $L$ on the domain $\{ f\in L^2(U,\mu)\mid L(f)\in L^2(U,\mu) \}$. Notice that such a localization procedure yields a closed operator on $L^2(U,\mu)$.

We claim that $\partial_t-L\big|_U$ possesses a well-defined fundamental solution $p^U(t,x,y)\in C^\infty(\R^+\times U\times U)$ on its own. To see this, let us notice that
	\be\label{eq:diss}
		\langle L\big|_Uf,f \rangle=\langle (X_0-\frac{1}{2}X_1^*X_1)(f),f \rangle=-\frac{1}{2}\|X_1f\|^2-\frac{1}{2}\langle f,\mathrm{div}_{\mu}(X_0)f \rangle,
	\ee
	for every $f\in C^\infty_c(U)$, seen as a subset of $L^2(U,\mu)$, and where $\langle\cdot,\cdot\rangle$ and $\|\cdot\|$ are, respectively, the inner product and the norm on $L^2(U,\mu)$.
	Since we suppose the vector fields $X_0$, $X_1$ bounded and with bounded derivatives, we deduce that there exists $c\geq 0$, not depending on $f$, such that
	\[
		\langle (L\big|_U-c\mathrm{Id})f,f \rangle\leq 0,
	\]
	and thus the operator $L\big|_U-c\mathrm{Id}$ is dissipative. The same holds for its adjoint.

		Since the operator $L\big|_U-c\mathrm{Id}$ is closed and dissipative on $L^2(U,\mu)$, as a consequence of the Lumer-Phillips theorem (see, for example \cite{EngNag,Pazy}), the operator $L\big|_U$ generates a strongly continuous semigroup $(e^{tL\big|_{U}})_{t\geq 0}$ on $L^2(U,\mu)$. The heat kernel $p^U$ is then defined as the Schwartz kernel for of the operator $e^{tL\big|_{U}}$. Finally, as the localized operator $L\big|_U$ still satisfies the H\"ormander hypoellipticity condition \eqref{eq:Hor}, we deduce that $p^U\in C^\infty(\R^+\times U\times U)$ as desired.

The following result asserts that $p(t,x,y)=p^U(t,x,y)+O(t^\infty)$ on $U$. This is an instance of the so-called \emph{principle of not feeling the boundary} \cite{kacfeeling, hsufeeling}. The proof presented here was suggested to us from Emmanuel Tr\'elat and is inspired by previous arguments in \cite{jerisonsanchez}, which use only the H\"ormander condition to establish the locality of the heat kernel on the diagonal.

This says in particular, if one is interested in a finite Taylor expansion of $p(t,x_{0},x_{0})$, it is not restrictive to work on a neighborhood of $x_{0}$. 
	\begin{prop}\label{prop:convoncomp}
		Let $x_0\in M$ and $U\subset M$ be a neighborhood of $x_0$ such that  
		\[
			\mathrm{dim}\left(\mathrm{span}_x\{X_1,[X_0,X_1]\}\right)=2,\quad\textrm{for every }x\in U.
		\] 
		Then, for all $(l,\alpha,\beta)\in\N\times \N^2\times \N^2$, every compact set $K\subset U$ and every $N\in\N\setminus\{0\}$, there exist $t_0>0$ and $c_N>0$ such that $$\big|\left(\partial_t^l\partial_x^\alpha\partial_y^\beta(p-p^U)\right)(t,x,y)\big|\leq c_Nt^N,$$ for every $0<t<t_0$, $(x,y)\in K\times K$.
	\end{prop}
	\begin{proof}\footnote{This argument is borrowed from the ongoing work \cite{CHT2}, with the agreement of the authors.}
		Extend at first both $p$ and $p^U$ by $0$ if $t<0$.
		Define $g(t,x,y):=p(t,x,y)-p^U(t,x,y)$, for $(t,x,y)\in \R\times U\times U$. Notice that $g$ is smooth on $\R^+\times U\times U$. We employ the notation $L^x$ or $L^y$, meaning that we regard $L$ as a differential operator acting either on the $x$ or on the $y$ variable respectively. By the properties defining a fundamental solution (compare with Definition~\ref{def:fundsol}), for any fixed $y\in U$ we deduce that $(\partial_t-L^x)(g(t,x,y))=0$ on $\R\times U$ in the sense of distributions. Since as soon as $p(t,x,y)$ is a fundamental solution to $L^x$, then so is $p(t,y,x)$ to the operator $L^y$, the same argument proves that, for any fixed $x\in U$, $(\partial_t-L^y)(g(t,y,x))=0$ on $\R\times U$. Defining $Y:=L^x+L^y-2\partial_t$, we infer that $Y(g)=0$. Also, since $L^x$ and $L^y$ satisfy the H\"ormander condition \eqref{eq:Hor} on $M$, the operator $Y$ satisfies the analogue condition in the product space $M\times M$. In particular $Y$ is hypoelliptic and  $g$ is smooth. 
		Since it vanishes, together with its derivatives, for $t<0$, the same happens at $t=0$ thus completing the proof.
	\end{proof}

As a consequence of Proposition~\ref{prop:convoncomp}, it is then not restrictive for what concerns the finite-order, small-time asymptotics of $p$, to work within $U$ with the localized version of \eqref{eq:Kol}. For the sake of clarity, with a slight abuse of notation we will omit the restriction sign in the sequel. 
 The weighted dilation $\delta_{\varepsilon}$,  
ensuring for $X_0$ twice the weight of $X_1$ as explained in Section~\ref{sec:FundSol}, is given by
\be
	\delta_{\varepsilon}(x_1,x_2)=(\varepsilon x_1,\varepsilon^3 x_2).
\ee
The volume $\mu$ rescales then by a factor $\varepsilon^4$, and, by choosing $\gamma=2$ in Proposition~\ref{prop:rescaling}, a fundamental solution $q_\varepsilon(t,x,y)$ of the operator
\be
	\partial_t-L_\varepsilon:=\partial_t-\varepsilon^2\left(\delta_{1/\varepsilon*}X_0+\frac{1}{2}(\delta_{1/\varepsilon*}X_1)^2\right),\quad 0<\varepsilon\leq 1,
\ee
is equal to $q_\varepsilon(t,x,y)=\varepsilon^4 p(\varepsilon^2 t,\delta_\varepsilon x,\delta_\varepsilon y)$. In particular, this construction also ensures that the operator $L_0$, appearing as the principal part of the limit $\lim_{\varepsilon\to 0^+}L_\varepsilon$ is hypoelliptic, with $q_0$ as a fundamental solution (cf.~also with \eqref{eq:X0}).
		
Using \eqref{eq:gradedcoord}, we can expand $L_\varepsilon$ as a polynomial series with respect to $\varepsilon$ as follows
\be
	L_\varepsilon=L_0+\varepsilon \mc{X}+\varepsilon^2\mc{Y}+\mc{Z}_{\varepsilon^3},
\ee
where
\begin{align}
	L_0&=\left(x_1(\partial_{x_1}\alpha_2)(0)\partial_{x_2}+\frac{1}{2}\partial^{2}_{x_1}\right),\\
	\mathcal{X}&=\left((\alpha_1)(0)\partial_{x_1}+\frac{1}{2}x_1^2(\partial^2_{x_1}\alpha_2)(0)\partial_{x_2}-\frac{1}{2}\frac{(\partial_{x_1}^2\alpha_2)(0)}{(\partial_{x_1}\alpha_2)(0)}\partial_{x_1}\right),\\
	\mathcal{Y}&=\left(x_1(\partial_{x_1}\alpha_1)(0)\partial_{x_1}+x_2(\partial_{x_2}\alpha_2)(0)\partial_{x_2}+\frac{1}{6}x_1^3(\partial^3_{x_1}\alpha_2)(0)\partial_{x_2}\right.\\&\qquad \left.-\frac{x_1}{2}\left(\frac{(\partial_{x_1}^3\alpha_2)(0)}{(\partial_{x_1}\alpha_2)(0)}-\left(\frac{(\partial_{x_1}^2\alpha_2)(0)}{(\partial_{x_1}\alpha_2)(0)}\right)^2\right)\partial_{x_1}\right).
\end{align}
Moreover, the term $\mc{Z}_{\varepsilon^3}$ can be presented as a smooth vector field on $U$ in the form $$\mc{Z}_{\varepsilon^3}=\varepsilon^3\left(\beta_1^\varepsilon(x_1,x_2)\partial_{x_1}+\beta_2^\varepsilon(x_1,x_2)\partial_{x_2}\right),$$ where $\beta_1^\varepsilon$ and $\beta_2^\varepsilon$ are smooth functions on $U$ that remain uniformly bounded for all values $0< \varepsilon\leq 1$. 
The remaining  point, as anticipated in Section~\ref{subsec:conv}, will be to show that the sequence of fundamental solutions $q_\varepsilon$, $0<\varepsilon\leq 1$, associated with the (localized) rescaled operators $\partial_t-L_\varepsilon$, properly converges to the fundamental solution $q_0$ of $\partial_t-L_0$. We show in Proposition~\ref{prop:estr} that, in our case, this follows explicitly from the computations and the Gaussian nature of $q_0$,  which permits in particular to control all its derivatives.

\subsection{Linear-quadratic operators}
The principal part $\partial_t-L_0$ of the operator $\partial_t-L_\varepsilon$ is an hypoelliptic operator of the form
\be
	\label{eq:linquad}
	\partial_t\varphi-\sum_{i=1}^n(Ax)_i\partial_{x_i}\varphi-\frac{1}{2}\sum_{j,k=1}^n(BB^*)_{jk}\partial^2_{x_j,x_k}\varphi,\quad \varphi\in \mc{D}'(\R\times\R^n), 
\ee
where $A$ is a real $n\times n$ real matrix, $B$ is a $n\times k$ matrix, with $k\leq n$, and whose hypoellipticity is guaranteed by \emph{Kalman's controllability condition}
\be
	\label{eq:Kalman}
	\mathrm{rank}\left[B,AB,\dots, A^{n-1}B\right]=n.
\ee
Its fundamental solution is classical in the literature and has the form of a Gaussian density as presented below.
The explicit expression of the fundamental solution is classical, while the coefficients of the full small-time heat kernel asymptotic and their relations with curvature-like invariants have been considered in \cite{BP15}.
\begin{prop}
\label{prop:fundsol} A fundamental solution  with respect to the Lebesgue measure on $\R^n$ of the differential operator \eqref{eq:linquad} is given by
\be 
	q_0(t,x,y)=\frac{e^{-\frac{1}{2}(y-e^{tA}x)^*G_t^{-1}(y-e^{tA}x)}}{(2\pi)^{\frac{n}{2}}\sqrt{\det(G_t)}},\quad\textrm{with}\quad G_t=e^{tA}\int_{0}^{t}e^{-\tau A}BB^*e^{-\tau A^*}d\tau\, e^{tA^*}.
\ee
\end{prop}

\begin{remark}\label{rem:kal}Observe that Kalman's condition \eqref{eq:Kalman} ensures indeed that $G_t$ is invertible, as soon as $t>0$. For future purposes we also define the matrix
	\be
	\label{eq:gamma}
	\Gamma_t=\left(\int_0^te^{-A \tau}BB^*e^{-A^*\tau}d\tau\right).
	\ee
Finally, notice that for our problem we have
\be
	A=\left(\begin{array}{cc}0 & 0 \\
	\partial_{x_1}\alpha_2(0) & 0\end{array}\right),\quad B=\left(\begin{array}{c}1\\ 0\end{array}\right),\ee
	and that 
	$\det(\Gamma_t)=\det(G_t)$ for all $ t\in[0,+\infty)$.
 
\end{remark}

\subsection{Main terms of the expansion}

Throughout this section we denote $S=\partial_{x_1}\alpha_2(0)$, which is non-zero by the assumption \eqref{eq:kkk}.

Iterating Duhamel's formula \eqref{eq:duh} three times, we obtain the following expression for $q_\varepsilon(1,0,0)$, namely:
\be\label{eq:duhex}
	q_\varepsilon(1,0,0)=q_0(1,0,0)+\varepsilon(q_0*\mc{X}q_0)(1,0,0)+\varepsilon^2(q_0*\mc{X}q_0*\mc{X}q_0+q_0*\mc{Y}q_0)(1,0,0)+\mc{R}_{\varepsilon^3}.
\ee
The leading term $q_0(1,0,0)$ is simply evaluated, with the notations of Proposition~\ref{prop:fundsol} by the formula
\be
q_0(1,0,0)=\frac{1}{2\pi\sqrt{\det(G_1)}},\quad\mathrm{where }\quad\det(G_1)=\det\left(\begin{array}{cc} 1 & \frac{S}{2}\\ \frac{S}{2} & \frac{S^2}{3}\end{array}\right)=\frac{S^2}{12}.
\ee
The computation of the convolutions
\be
(q_0*\mc{X}q_0)(1,0,0),\quad(q_0*\mathcal{X}q_0*\mathcal{X}q_0)(1,0,0)\quad\mathrm{ and }\quad(q_0*\mathcal{Y}q_0)(1,0,0),
\ee
needs a technical lemma which is crucial in what follows.

\begin{lemma}\label{lemma:Gaussian}
	For every pair $(z,w)\in\R^2\times \R^2$, $0\leq s\leq 1$ and $0\leq r\leq 1-s$, the following identity holds
	\be
	\label{eq:gausid}
	\frac{e^{-\frac{1}{2}(w-e^{rA}z)^*G_r^{-1}(w-e^{rA}z)}}{2\pi\sqrt{\det(G_r)}}\cdot\frac{e^{-\frac{1}{2}w^*\Gamma_{1-s-r}^{-1}w}}{2\pi\sqrt{\det(G_{1-s-r})}}=\frac{e^{-\frac{1}{2}z^*\Gamma_{1-s}^{-1}z}}{2\pi\sqrt{\det(\Gamma_{1-s})}}\cdot\frac{e^{-\frac{1}{2}(w-\nu)^*\Sigma^{-1}(w-\nu)}}{2\pi\sqrt{\det(\Sigma)}},
	\ee
	with
	$\Sigma=\Gamma_{1-s-r}e^{-rA^*}\Gamma_{1-s}^{-1}\Gamma_{r}e^{rA^*}$ and
$\nu=\Gamma_{1-s-r}e^{-rA^*}\Gamma_{1-s}^{-1}z$.

\end{lemma}
\begin{proof}
	The product of two Gaussian densities is again a Gaussian density, according to the following classical identity
	\be
	(v-a)^*\mc{A}(v-a)+(v-b)^*\mc{B}(v-b)=(v-c)^*(\mc{A}+\mc{B})(v-c)+(a-b)^*\mc{C}(a-b),
	\ee
	which holds for every $v,a,b\in \R^2$, $\mc{A},\mc{B}$ square $2\times 2$ real matrices, and for $\mc{C}:=(\mc{A}^{-1}+\mc{B}^{-1})^{-1}\in \mathrm{M}_2(\R)$ and $c:=(\mc{A}+\mc{B})^{-1}(\mc{A}a+\mc{B}b)\in\R^2$.
	
	In our case, we apply this identity with (compare with Remark~\ref{rem:kal})
	\be
	\mc{A}:=\Gamma_{1-s-r}^{-1},\quad  \mc{B}:=G^{-1}_r,\quad a:=0,\quad  b:=e^{rA}z,	
	\ee
	and the only non-trivial identity to prove is the following:
	\begin{align}\label{eq:key}
	\mc{A}^{-1}+\mc{B}^{-1}&=G_r+\Gamma_{1-s-r}\\&=e^{rA}\left(\Gamma_r+e^{-rA}\Gamma_{1-s-r}e^{-rA^*}\right)e^{rA^*}\\
	&=e^{rA}\left(\Gamma_r+e^{-rA}\int_r^{1-s}e^{(-\tau+r)A}BB^*e^{(-\tau+r)A^*}d\tau\, e^{-rA^*}\right)e^{rA^*}\\
	&=e^{rA}\Gamma_{1-s}e^{rA^*}.
	\end{align}
	Indeed, once \eqref{eq:key} is established, one computes $\Gamma_{1-s-r}^{-1}+G_r^{-1}=\mc{A}+\mc{B}$ using the identity 
	$\mc{A}+\mc{B}=\mc{B}\left(\mc{A}^{-1}+\mc{B}^{-1}\right)\mc{A}$.
	The claim then follows from a routine computation.
\end{proof}

\begin{prop}
	\label{prop:est}
	The following list of equalities holds true:
	\begin{align}
	(q_0*\partial_{x_1}q_0)(1,0,0)&=0,\\
	(q_0*x_1^2\partial_{x_2}q_0)(1,0,0)&=0,\\
	(q_0*\partial_{x_1}q_0*\partial_{x_1}q_0)(1,0,0)&=\frac{1}{2\pi\sqrt{\det(G_1)}}\left(-\frac{1}{2}\right),\\
	(q_0*x_1^2\partial_{x_2}q_0*\partial_{x_1}q_0)(1,0,0)&=\frac{1}{2\pi\sqrt{\det(G_1)}}\left(\frac{3}{14S}\right),\\
	(q_0*\partial_{x_1}q_0*x_1^2\partial_{x_2}q_0)(1,0,0)&=\frac{1}{2\pi\sqrt{\det(G_1)}}\left(-\frac{3}{14S}\right),\\
	(q_0*x_1\partial_{x_1}q_0)(1,0,0)&=\frac{1}{2\pi\sqrt{\det(G_1)}}\left(-\frac{1}{2}\right),\\
	(q_0*x_2\partial_{x_2}q_0)(1,0,0)&=\frac{1}{2\pi\sqrt{\det(G_1)}}\left(-\frac{1}{2}\right),\\
	(q_0*x_1^3\partial_{x_2}q_0)(1,0,0)&=\frac{1}{2\pi\sqrt{\det(G_1)}}\left(-\frac{3}{14S}\right),\\
	(q_0*x_1^2\partial_{x_2}q_0*x_1^2\partial_{x_2}q_0)(1,0,0)&=\frac{1}{2\pi\sqrt{\det(G_1)}}\left(\frac{9}{70S^2}\right).\\
	\end{align}
	In particular, we have that 
	\begin{align}
	\label{eq:exp}
	(q_0*\mc{X}q_0)(1,0,0)&=0,\\
	(q_0*\mc{X}q_0*\mc{X}q_0+q_0*\mathcal{Y}q_0)(1,0,0)&=\frac{1}{2\pi\sqrt{\det(G_1)}}\left(-\frac{1}{2}\alpha_1^2(0)-\frac{12}{35}\left(\frac{\partial_{x_1}^2\alpha_2(0)}{\partial_{x_1}\alpha_2(0)}\right)^2\right.\\&\left.-\frac{1}{2}\partial_{x_1}\alpha_1(0)-\frac{1}{2}\partial_{x_2}\alpha_2(0)+\frac{3}{14}\frac{\partial_{x_1}^3\alpha_2(0)}{\partial_{x_1}\alpha_2(0)}-\frac{1}{2}\alpha_1(0)\frac{\partial_{x_1}^2\alpha_2(0)}{\partial_{x_1}\alpha_2(0)}\right).
	\end{align}
\end{prop}

\begin{proof}
	We show as an example how to find the value of $(q_0*\partial_{x_1}q_0)(1,0,0)$ and $(q_0*\partial_{x_1}q_0*\partial_{x_1}q_0)(1,0,0)$, the other cases being all similar. In the following we tacitly assume that the values of all the moments up to the sixth order of a Gaussian distribution are known, and we constantly use the equality
	\be\label{eq:der}
		\partial_{x_i}q_0(t,x,y)=-\left[\Gamma_t^{-1}\left(x-e^{-A t }y\right)\right]_iq_0(t,x,y),\quad\textrm{for}\quad i=1,2.
	\ee
	Here $[v]_{i}$ for $i=1,2$ denotes the $i$-th component of a vector $v\in \R^{2}$. Working in a local chart, it is not restrictive to evaluate all the integrals appearing throughout the proof as if they were on $\R^2$.
	For the first one, using Lemma~\ref{lemma:Gaussian}, we have
	\begin{align}
	(q_0*\partial_{x_1}q_0)(1,0,0)&=\int_0^1\int_{\R^2}q_0(s,0,z)\partial_{x_1}q_0(1-s,z,0)dzds\\
	&=-\int_0^1\int_{\R^2}\left[\Gamma_{1-s}^{-1}\left(z\right)\right]_1\frac{e^{-\frac{1}{2}z^*G_s^{-1}z}}{2\pi\sqrt{\det(G_s)}}
	\frac{e^{-\frac{1}{2}z^*\Gamma_{1-s}^{-1}z}}{2\pi\sqrt{\det(G_{1-s})}}dzds=0,
	\end{align}
	since the product of two Gaussian densities of zero mean is again a Gaussian density of zero mean, $\left[\Gamma_{1-s}^{-1}\left(z\right)\right]_1$ is a polynomial of degree $1$ in $z$, and all the centered moments of odd order of a zero-mean Gaussian are zero.
	
	For the second one we proceed as follows: by Lemma~\ref{lemma:Gaussian} we have the identity
	\be
	q_0(r,z,w)q_0(1-s-r,w,0)=\frac{e^{-\frac{1}{2}z^*\Gamma_{1-s}^{-1}z}}{2\pi\sqrt{\det(\Gamma_{1-s})}}\cdot\frac{e^{-\frac{1}{2}(w-\nu)^*\Sigma^{-1}(w-\nu)}}{2\pi\sqrt{\det(\Sigma)}},
	\ee
	with
	$\Sigma:=\Gamma_{1-s-r}e^{-rA^*}\Gamma_{1-s}^{-1}\Gamma_{r}e^{rA^*}$
	and $\nu:=\Gamma_{1-s-r}e^{-rA^*}\Gamma_{1-s}^{-1}z$.
	
	On a second step, using similar arguments to those exploited above, one computes the product
	\be
	\frac{e^{-\frac{1}{2}z^*G_s^{-1}z}}{2\pi\sqrt{\det(G_{s})}}\cdot\frac{e^{-\frac{1}{2}z^*\Gamma_{1-s}^{-1}z}}{2\pi\sqrt{\det(G_{1-s})}}=
	\frac{1}{2\pi\sqrt{\det(\Gamma_1)}}\cdot\frac{e^{-\frac{1}{2}z^*\Sigma'^{-1}z}}{2\pi\sqrt{\det(\Sigma')}},
	\ee
	where 
	$\Sigma':=\Gamma_{1-s}e^{-sA^*}\Gamma_1^{-1}\Gamma_se^{sA}$.
	
	From here we deduce (for $\Sigma$, $\nu$ and $\Sigma'$ defined as above):
	\begin{align}
	&(q_0*\partial_{x_1}q_0*\partial_{x_1}q_0)(1,0,0)=
	\\&\int_{0}^1\int_{\R^2}q_0(s,0,z)\int_{0}^{1-s}\int_{\R^2}\partial_{x_1}q_0(r,z,w)\partial_{x_1}q_0(1-s-r,w,0)dwdrdzds\\
	&=\int_{0}^1\int_{\R^2}q_0(s,0,z)\frac{e^{-\frac{1}{2}z^*\Gamma_{1-s}^{-1}z}}{2\pi\sqrt{\det(G_{1-s})}}\times\\
	&\hphantom{\int_{0}^1\int_{\R^2}q_0(s,0,z)}\times\int_{0}^{1-s}\int_{\R^2}\left[\Gamma_r^{-1}\left(z-e^{-rA}w\right)\right]_1\left[\Gamma_{1-s-r}^{-1}w\right]_1\frac{e^{-\frac{1}{2}(w-\nu)^*\Sigma^{-1}(w-\nu)}}{2\pi\sqrt{\det(\Sigma)}}dwdrdzds\\
	&=\int_0^1\int_{\R^2}\frac{e^{-\frac{1}{2}z^*G_s^{-1}z}}{2\pi\sqrt{\det(G_{s})}}\cdot\frac{e^{-\frac{1}{2}z^*\Gamma_{1-s}^{-1}z}}{2\pi\sqrt{\det(G_{1-s})}}Q(z)dzds\\&=\frac{1}{2\pi\sqrt{\det(\Gamma_1)}}\int_0^1\int_{\R^2}\frac{e^{-\frac{1}{2}z^*\Sigma'^{-1}z}}{2\pi\sqrt{\det(\Sigma')}}Q(z)dzds,
	\end{align}
	where $Q(z)$ is a polynomial in $z$ that can be determined explicitly and whose coefficients are rational functions in $1-s$. Their expression depends on the moments of even order of a Gaussian density. 
		To arrive to the results of the proposition is now just a routine computation.
\end{proof}

\subsection{Bounds on the remainder term} 
 Iterating three times Duhamel's formula, and collecting all the terms containing at least the factor $\varepsilon^3$, wee see that the remainder term $\mc{R}_{\varepsilon^{3}}$ in \eqref{eq:duhex} is a finite sum of items of these three kinds:
\begin{itemize}\setlength\itemsep{1em}
	\item [a)] $(q_0*\varepsilon^{\alpha_1}\mc{D}_1q_0)(1,0,0)$, with $\alpha_1\geq 3$, or
	\item [b)] $(q_0*\varepsilon^{\beta_1}\mc{D}_1q_0*\varepsilon^{\beta_2}\mc{D}_2q_0)(1,0,0)$, with $\beta_1+\beta_2\geq 3$, or
	\item [c)]
	$(q_\varepsilon*\varepsilon^{\gamma_1}\mc{D}_1q_0*\varepsilon^{\gamma_2}\mc{D}_2q_0*\varepsilon^{\gamma_3}\mc{D}_3q_0)(1,0,0)$, with $\gamma_1+\gamma_2+\gamma_3\geq 3$.
\end{itemize}
Here the $\mc{D}_i$ are chosen among a finite set of 
vector fields which are all supported and bounded on $U$, uniformly with respect to the value of $\varepsilon\in (0,1]$. 

The objective of this section is to show that the ratio $\mc{R}_{\varepsilon}^3/\varepsilon^3$ remains uniformly bounded in $\varepsilon$, for such values of the parameter. To this end, it suffices to control, uniformly with respect to $\varepsilon$, the absolute values of all the terms appearing in each of the three families.

For what concerns summands belonging either to a) or b), the result follows easily applying the same kind of computations exploited in Proposition~\ref{prop:est}. For items appearing in c), their uniform estimate is a consequence of the following result.

\begin{prop}
	\label{prop:estr}
	Let $f_1,f_2,f_3$ be smooth, bounded functions supported on $U$, and $i,j,k\in\{1,2\}$. Then there exists  $C>0$, depending on all the previous data, such that
	\be
		\label{eq:fundest}
		\left|\left(q_\varepsilon*f_1(x_1,x_2)\partial_{x_i}q_0*f_2(x_1,x_2)\partial_{x_j}q_0*f_3(x_1,x_2)\partial_{x_k}q_0\right)(1,0,0)\right|\leq C.
	\ee
\end{prop}
\begin{proof}
	Again all calculations will be carried out on $\R^2$, our statement being local. Begin with the obvious inequality (recall that $q_\varepsilon\geq 0$, being a probability density)
	\begin{align}
	\label{eq:start}
		\\
		&\left|\left(q_\varepsilon*f_1(x_1,x_2)\partial_{x_i}q_0*f_2(x_1,x_2)\partial_{x_j}q_0*f_3(x_1,x_2)\partial_{x_k}q_0\right)(1,0,0)\right|\leq\\
		&C_1\int_0^1\int_{\R^2}q_\varepsilon(s,0,z)\bigg(\int_{0}^{1-s}\int_{\R^2}\left|\partial_{x_i}q_0(r,z,w)\right|\bigg(\int_{0}^{1-s-r}\int_{\R^2}\left|\partial_{x_j}q_0(l,w,t)\right|\times\\ &\hphantom{xxxxxxxxxxx}\times\left|\partial_{x_k}q_0(1-s-r-l,t,0)\right|dtdl\bigg)dwdr\bigg)dzds,
	\end{align}
	with $C_1:=\sup_U|f_1|\cdot\sup_U|f_2|\cdot\sup_U|f_3|$. Using Lemma~\ref{lemma:Gaussian} iteratively together with an argument close to the one contained in the proof of Proposition~\ref{prop:est}, one can treat all the nested integrals in \eqref{eq:start}, and majorize them by
	\be
		\label{eq:maj}
		C_1\int_0^1\int_{\R^2}q_\varepsilon(s,0,z)P(|z_1|,|z_2|)\frac{e^{-\frac{1}{2}z^*\Gamma_{1-s}^{-1}z}}{2\pi\sqrt{\det(\Gamma_{1-s})}}dzds,
	\ee
	where $P$ is a polynomial in $|z_1|,|z_2|$ whose coefficients are rational functions of $(1-s).$ Notice that $P$ can be explicitly computed in terms of the absolute moments of a Gaussian distribution. 
	We now claim that there exists a positive constant $C_2$ such that
	\be\label{eq:prod}
		P(|z_1|,|z_2|)\frac{e^{-\frac{1}{2}z^*\Gamma_{1-s}^{-1}z}}{2\pi\sqrt{\det(\Gamma_{1-s})}}\leq C_2\quad\mathrm{on }\quad[0,1]\times \R^2.
	\ee
	In fact, recalling the notation $S=\partial_{x_1}\alpha_2(0)$, we have $$e^{-\frac{1}{2}z^*\Gamma_{1-s}^{-1}z}=e^{\frac{2}{S^2(s-1)^3}(S^2(s-1)^2z_1^2-3S(s-1)z_1z_2+3z_2^2)}.$$ In particular, as $s\nearrow 1$, $e^{-\frac{1}{2}z^*\Gamma_{1-s}^{-1}z}\sim e^{\frac{6 z_2^2}{S^2(s-1)^3}}$ on the set $\{z_2\neq 0\}$, while $e^{-\frac{1}{2}z^*\Gamma_{1-s}^{-1}z}\sim e^{\frac{2 z_1^2}{(s-1)}}$ on $\{z_2=0\}$. Then \eqref{eq:prod} remains uniformly bounded on $[0,1]\times \R^2$ as claimed.
	
The integral of $q_{\varepsilon}(s,0,\cdot)$ is bounded independently on $\varepsilon$, since $q_{\varepsilon}$ is a probability density. Hence, we conclude that
	\be
		C_1\int_0^1\int_{\R^2}q_\varepsilon(s,0,z)P(|z_1|,|z_2|)\frac{e^{-\frac{1}{2}z^*\Gamma_{1-s}^{-1}z}}{2\pi\sqrt{\det(\Gamma_{1-s})}}dzds\leq C_1\cdot C_2,
	\ee
	and the proposition follows with $C=C_1\cdot C_2$.
\end{proof}

We conclude this section collecting the results of Propositions~\ref{prop:est} and \ref{prop:estr}, to obtain an explicit formula for $q_\varepsilon(1,0,0)$ in \eqref{eq:duhex}.
	\begin{thm}\label{thm:fc}
		The fundamental solution $q_\varepsilon(1,0,0)$ admits the following asymptotic expansion
		\begin{align}
			q_\varepsilon(1,0,0)=\frac{1}{2\pi \sqrt{\det(G_1)}}&\left(1+\varepsilon^2\left(-\frac{1}{2}\alpha_1^2(0)-\frac{1}{2}\bigg(\partial_{x_1}\alpha_1(0)+\partial_{x_2}\alpha_2(0)\right.\right.\\
			&-\alpha_1(0)\frac{\partial_{x_1}^2\alpha_2(0)}{\partial_{x_1}\alpha_2(0)}\bigg)\left.\left.-\frac{12}{35}\left(\frac{\partial_{x_1}^2\alpha_2(0)}{\partial_{x_1}\alpha_2(0)}\right)^2+\frac{3}{14}\frac{\partial_{x_1}^3\alpha_2(0)}{\partial_{x_1}\alpha_2(0)}\right)+o(\varepsilon^2)\right).
		\end{align}
	\end{thm}

\section{Curvature-like invariants}\label{sec:Curv}

Let $M$ be a two-dimensional smooth and connected manifold, and $X_0,X_1$ be a pair of smooth vector fields on $M$, such that  for every $x\in M$
\begin{equation} \label{eq:ipotesi}
\mathrm{span}\{X_1(x), [X_0,X_1](x)\}=T_xM.
\end{equation}
Consider the differential system
\be\label{eq:contprob}
	\dot{x}=X_0(x)+uX_1(x),\quad x(0)=x_0\in M,\quad u\in\R,
\ee
and let $T>0$. An \emph{admissible control} $t\mapsto u(t)$ is an element of $L^2([0,T],\R)$ such that the corresponding trajectory $x_u(\cdot)$, solution to \eqref{eq:contprob}, is defined on $[0,T]$. The set of admissible controls $\Omega^{T}_{x_0}\subset L^2([0,T],\R)$ is an open set. If $u\in\Omega^{T}_{x_0}$, we say that $x_u(\cdot)$ is an \emph{admissible trajectory}, and we define the \emph{attainable set} from $x_0$ at time $T$ as $\Att=\{x_u(T)\mid u\in\Omega_{x_0}\}$. By Sussmann-Jurdjevic Theorem \cite{SJ72,jurdjevicbook}, $\Att$ has non-empty interior thanks to our assumptions.

Introduce now the energy functional $J_T:L^2([0,T],\R)\to\R$, defined by
\be
	\label{eq:energy}
	J_T(u)=\frac{1}{2}\int_0^T|u(t)|^2dt.
\ee
For a fixed final point $x\in M$, we are interested into solving the optimal control problem
\be
	\label{eq:optcontrprob}
	S_{x_0}^T(x)=\inf\left\{ J_T(u)\,\mid\,u\in \Omega_{x_{0}}^{T}, \, x_u(T)=x\right\},
\ee
with the convention that $S_{x_0}^T(x)=+\infty$ if $x\not\in \Att$. The geodesic flow of this problem can be seen as a flow on $T^*M$, associated with the Hamiltonian
\be
	H(p,x)=\langle p,X_0(x)\rangle+\frac{1}{2}\langle p, X_1(x)\rangle^2,\quad (p,x)\in T^*M.
\ee
Then Hamilton's equations\footnote{By means of the canonical symplectic form $\sigma$ on $T^*M$ we define, for every $f:T^*M\to \R$ and for every $\lambda\in T^*M$, the Hamiltonian lift $\vec{f}$ of $f$ through the identity $\sigma_\lambda(\cdot,\vec{f})=d_\lambda f(\cdot)$.} are written in the form $\dot{\lambda}=\vec{H}(\lambda)$, where $\lambda=(p,x)\in T^*M$ and $\vec{H}$ is the Hamiltonian vector field associated with $H$. Integral curves of $\vec{H}$ are called \emph{extremals}, and their projections
on the manifold permit to find the so-called \emph{normal} geodesics, whose short enough pieces solve the optimal problem \eqref{eq:optcontrprob} between their end-points. It is well-known that, under the assumption \eqref{eq:ipotesi}, the optimal control problem has no singular minimizers, i.e., every solution to the optimal problem satisfy this necessary condition. Moreover, the control $u$ associated with a normal geodesic, called also an \emph{optimal control}, is smooth on $[0,T]$. A detailed presentation on the subject may be found, for example, in \cite{agrachevbook}.

\subsection{The canonical frame} Given a geodesic $t\mapsto x(t)$, consider the \emph{geodesic flag} associated with $x(t)$ \be\label{eq:flag}
	\mc{F}_{x(t)}^1\subset\mc{F}_{x(t)}^2\subset\dotso\subset T_{x(t)}M,\quad t\in[0,T],
\ee
defined by  $\mc{F}_{x(t)}^i=\mathrm{span}\{X_1,\dotso,X_i\}\big|_{x(t)}$ for every $i\in\N$. Here 
$X_j=\mathrm{ad}(X_0+\alpha X_1)^{j-1}X_1$, and $\alpha\in C^\infty(M)$ is such that $\alpha(x(t))=u(t)$ for every $t\in[0,T]$. The vector fields $X_i$ are smooth along $x(t)$. We say that $x(\cdot)$ is \emph{equiregular} at $t$ if $\mathrm{dim}\mc{F}^i_{x(t)}$ is locally constant for every $i>0$, and \emph{ample} at $t$ if there exists $m>0$ such that $\mc{F}^m_{x(t)}=T_{x(t)}M$.

Under the assumption \eqref{eq:ipotesi}, it follows immediately that, in our setting, all geodesics are indeed ample and equiregular. For a comprehensive presentation of these notions in the general setting the reader is referred to \cite[Chapter 3]{MemAMS}.

If $x(t)$ is ample and equiregular, then the geodesic flag terminates with $\mc{F}^2_{x(t)}=T_{x(t)}M$ for every $t\in[0,T]$, and is described with the basis $X_1$ and $X_2=[X_1,X_0]$ along $x(t)$. Then the following result, which is a specification of a more general statement adapted to our context, holds true  (cf.~also \cite[Chapter 7]{MemAMS}).

\begin{prop}[Zelenko-Li, \cite{lizel}] \label{prop:cf}
	Assume that $\lambda(t)$ is the lift of an ample and equiregular geodesic $x(t)$. Then, there exists a smooth moving frame $\{E_1,F_1,E_2,F_2\}$ along $\lambda(t)$ such that, for every $t\in[0,T]$,
	\begin{itemize}
		\item [i)]$\pi_*(E_i)\big|_{\lambda(t)}=0$, where $\pi:T^*M\to M$ is the canonical projection.
		\item
		[ii)] $\sigma(E_i,E_j)\big|_{\lambda(t)}=\sigma(F_i,F_j)\big|_{\lambda(t)}=0$, $\sigma(E_i,F_j)\big|_{\lambda(t)}=\delta_{ij}$, where $\sigma$ is the canonical symplectic form defined on the cotangent space $T^*M$. In particular the canonical frame is a Darboux basis for $T_{\lambda(t)}T^*M$.
		\item [iii)] The structural equations satisfied by this frame are
		\begin{gather}
			\dot{E}_2(t)=E_1(t),\qquad \dot{E}_1(t)=-F_1(t),\\ \dot{F}_2(t)=R_{22}(t)E_2(t),\qquad \dot{F}_1(t)=R_{11}(t)E_1(t)-F_2(t).
		\end{gather}
	\end{itemize}
\end{prop}
The remaining part of this section will be thus devoted to the computation of the term $R_{22}(t)$.

\subsection{Structural equations for the canonical frame} \label{s:tricks}

Let us consider on $M$ the smooth vector fields $X_0,X_1$ and denote $X_2:=[X_0,X_1]$. Define  the associated fiber-wise linear functions $h_i:T^*M\to \R$ for $i=0,1,2$, defined by $h_i(p,x):=\langle p,X_i(x)\rangle$.

Owing to the identity $\{h_i,h_j\}(p,x)=\langle p,[X_i,X_j](x)\rangle$, the Lie bracket relations between $X_0, X_1$ and $X_2$ translate into Poisson bracket relations as follows:
\begin{align}
	\label{eq:PBracket}
		\{h_{0},h_{1}\}&=h_{2},\\
		\{h_{1},h_{2}\}&=c_{12}^{1}h_{1}+c_{12}^{2}h_{2},\\
		\{h_{0},h_{2}\}&=c_{02}^{1}h_{1}+c_{02}^{2}h_{2},
\end{align}
for some structural functions $c_{ij}^k\in C^{\infty}(M)$.

The functions $h_{1},h_{2}:T^{*}M\to \R$ define a coordinate system on the fibers of $T^*M$. Let us denote by $\partial_{h_1},\partial_{h_2}$ the corresponding vector fields. These vector fields are vertical, namely $\ker\pi_*=\mathrm{span}\{\partial_{h_1},\partial_{h_2}\}$, where $\pi_*:T(T^*M)\to TM$ is the differential of the canonical projection $\pi:T^*M\to M$. We complete $\{\partial_{h_1},\partial_{h_2}\}$ to a basis of $T(T^*M)$ by means of the Hamiltonian lifts $\vec{h}_1$ and $\vec{h}_2$.

In this formalism, the geodesic Hamiltonian $H$ reads $H=h_0+\frac{1}{2}h_1^2$, and its Hamiltonian lift becomes $\vec{H}=\vec{h}_0+h_1\vec{h}_1$. For every integral curve $t\mapsto \lambda(t)$, $t\in[0,T]$   of the vector field $\vec{H}$, and every smooth function $f:T^*M\to \R$, we have
\be
	\label{eq:diffFunc}
	\dot{f}(\lambda(t))=\vec{H}(f(\lambda(t)))=\{H,f\}(\lambda(t)),
\ee
the dot meaning the derivative with respect to $t$. On the other hand, for the vector fields $\vec{h}_i$, the differentiation along the geodesic lift becomes
\be
	\label{eq:diffVec}
	\dot{\vec{h}}_i(\lambda(t))=[\vec{H},\vec{h}_i](\lambda(t))=\overrightarrow{\{h_0,h_i\}}+h_1\overrightarrow{\{h_1,h_i\}}-\delta_{i1}\vec{h}_1,
\ee
the dot meaning here the Lie derivative with respect to $\vec H$. To compute the quantity  $\overrightarrow{\{h_i,h_j\}}$ one uses  the structural equations in \eqref{eq:PBracket}. Moreover we recall that for $g\in C^{\infty}(T^{*}M)$ fiber-wise constant, we have 
\be\label{eq:fwconstant}
	\dot{g}(\lambda(t))=\{H,g\}(\lambda(t))=(X_{0}g)(\lambda(t))+(X_{1}g)(\lambda(t))h_{1}(\lambda(t)),
\ee
where in the right hand side $g$ is treated as a function on the manifold.
To compute $\dot{\partial}_{h_i}$, for $i=1,2$, one has
\be
	\dot{\partial}_{h_i}(\lambda(t))=[\vec{h}_0+h_1\vec{h}_1,\partial_{h_i}](\lambda(t))=[\vec{h}_0,\partial_{h_i}](\lambda(t))+[\vec{h}_1,\partial_{h_i}](\lambda(t))-\delta_{i1}\vec{h_1}(\lambda(t)).
\ee
For the last term $[\vec{h}_j,\partial_{h_i}](\lambda(t))$ it is useful to observe that $[\vec{h}_1,\partial_{h_i}](\lambda(t))$ is vertical, that is $\pi_*[\vec{h}_1,\partial_{h_i}]=0$, hence one can write
\be
	[\vec{h}_j,\partial_{h_i}](\lambda(t))=[\vec{h}_j,\partial_{h_i}](h_1)\partial_{h_1}(\lambda(t))+[\vec{h}_j,\partial_{h_i}](h_2)\partial_{h_2}(\lambda(t)).
\ee

\subsection{Computations}
We provide in this sections all the main computations. For clarity's sake, we systematically drop the explicit dependence on $t$ along the extremal. The next two lemmas follows by direct computations as explained in Section~\ref{s:tricks}.
\begin{lemma}[Fundamental computations]\label{lemma:fund} Setting
\begin{align}
	A &:= c_{12}^{1}h_{1}+c_{02}^{1},\\
	B &:= c_{12}^{2}h_{1}+c_{02}^{2},\\
	C &:= c_{12}^{1}h_{1}+c_{12}^{2}h_{2},\\
	M_{i} &:=X_{i}(c_{02}^{1})h_{1}+X_{i}(c_{02}^{2})h_{2}+X_{i}(c_{12}^{1})h_{1}^{2}+X_{i}(c_{12}^{2})h_{1}h_{2},\quad i=1,2,
\end{align} 
along any extremal $\lambda(t)$, the following relations hold true:
	\begin{align}
		\dot{\partial}_{h_{1}}&= -\vec{h}_{1}-A\partial_{h_{2}},\\
		\dot{\partial}_{h_{2}}  & =  -\partial_{h_{1}}-B\partial_{h_{2}},\\
		\dot{\vec{h}}_{1} & = \vec{h}_{2}, \\
		\dot{\vec{h}}_{2} & = (A+C)\vec{h}_{1} +B\vec{h}_{2}-M_{1}\partial_{h_{1}}-M_{2}\partial_{h_{2}},\\
		\ddot{\partial}_{h_{2}}&=\vec{h}_{1}+B\partial_{h_{1}}+(A-\dot B+B^{2})\partial_{h_{2}},\\
		\dddot{\partial}_{h_{2}}&=\vec{h}_{2}-B\vec{h}_{1}-(A-2\dot B+B^{2})\partial_{h_{1}}
		+(\dot A-\ddot B-2AB+3B\dot B-B^{3})\partial_{h_{2}}.
	\end{align}
\end{lemma}

\begin{lemma}[Computation of symplectic products] \label{lemma:pro2}
	We have:
	\begin{align}
		\sigma(\partial_{h_{1}},\vec{h}_{1})&=\sigma(\partial_{h_{2}},\vec{h}_{2})=1,\\
		\sigma(\vec{h}_{1},\vec{h}_{2})&=\{h_1,h_{2}\}=C,\\
		\sigma(\partial_{h_{2}},\partial_{h_{2}})&=0,\\
		\sigma(\partial_{h_{2}},\dot \partial_{h_{2}})&=0,\\
		\sigma(\partial_{h_{2}},\ddot \partial_{h_{2}})&=0\textrm{ implies also }\sigma(\dot\partial_{h_{2}},\dot \partial_{h_{2}})=0,\\
		\sigma(\partial_{h_{2}},\dddot \partial_{h_{2}})&=1\textrm{ implies also }\sigma(\dot\partial_{h_{2}},\ddot \partial_{h_{2}})=-1,\\
		\sigma(\dot \partial_{h_{2}},\dddot \partial_{h_{2}})&=0\textrm{ implies also }
		\sigma( \ddot \partial_{h_{2}}, \ddot \partial_{h_{2}})= 0.
	\end{align}
\end{lemma}
We are now able to compute the coefficient $R_{22}$ in Proposition \ref{prop:cf}.
\begin{prop} Along any extremal $\lambda(t)$, the following relation holds true:
\begin{align}
	\label{eq:R22}  R_{22}&=h_{1}^{2}[-(c_{12}^{2})^{2}+3X_{1}(c_{12}^{2})]\\
	&+h_{1}[-3c_{12}^{1}-2c_{12}^{2}c_{02}^{2}+3X_{0}(c_{12}^{2})+3X_{1}(c_{02}^{2})]\\
	&+h_{2}[2c_{12}^{2}]\\
	&+(-2c_{02}^{1}-(c_{02}^{2})^{2}+3X_{0}(c_{02}^{2})).
\end{align}
\end{prop}
\begin{proof}
The vector $E_1$ is uniquely determined by the conditions
\begin{itemize}
	\item[(i)]$\pi_*E_1=0$,
	\item[(ii)]$\pi_*\dot{E}_1=0$,
	\item[(iii)] $\sigma(\ddot{E}_1,\dot{E}_1)=1$.
\end{itemize}
Conditions (i) and (ii) imply that $E_1 = f \partial_{h_{2}}$, so that $$\dot E_{1}=\dot f \partial_{h_{2}}+ f  \dot\partial_{h_{2}},\qquad \ddot E_{1}=\ddot f \partial_{h_{2}}+2\dot f \dot \partial_{h_{2}}+ f  \ddot\partial_{h_{2}}.$$
From the previous formulas we find $\sigma(\ddot{E}_1,\dot{E}_1)=f^{2}$, and thus condition (iii) finally implies that $f=1$. Notice that this means $E_{1}=\partial_{h_{2}}$, $E_{2}=\dot \partial_{h_{2}}$, $F_{2}=-\ddot \partial_{h_{2}}$,
and we eventually arrive to the expression
\be 
	R_{22}=\sigma(\dot F_{2},F_{2})=\sigma(\dddot{\partial}_{h_{2}},\ddot \partial_{h_{2}}).
\ee
By Lemmas \ref{lemma:fund} and \ref{lemma:pro2}, the latter expression reduces to
\be
	R_{22}=-C-2A-B^{2}+3\dot B.
\ee
Recalling the fact that \[\dot{B}=X_0(c_{02}^2)+(X_1(c_{02}^2)+X_0(c_{12}^2))h_1+c_{12}^2h_2+X_1(c_{12}^2)h_1^2,\] we conclude.
\end{proof}
\begin{defi}\label{d:kk}
We define $K_{1}$ and $K_{2}$ as the coefficients of $h_1^2$ and $h_2$, respectively, in the expression of $R_{22}$, that is 
\begin{align}\label{eq:coefffcc}
K_{1}&:=-(c_{12}^{2})^{2}+3X_{1}(c_{12}^{2}),\\
K_{2}&:=2c_{12}^{2}.
\end{align}
\end{defi}
The coefficient $K_1$ can be interpreted as the contraction along the controlled vector field $X_1$ of the second-order part of the curvature operator associated with the control problem, that is a sort of scalar curvature. Indeed, in the case of a control problem associated to Riemannian geodesic problem, this quantity coincides, up to a scalar factor, to the scalar curvature of the Riemannian metric.

The coefficient $K_2$ plays the role of a ``higher order'' curvature. Its presence here is necessary to give a correct interpretation since the drift $X_{0}$ is necessary to fulfill the H\"ormander condition. 
	
\section{Geometric interpretation of the small-time asymptotics}\label{sec:Coordinates}

The purpose of this section is to interpret the first coefficient appearing in the small-time asymptotics of the heat kernel $p(t,x_0,x_0)$, by means of the curvature-like invariants that can be deduced from \eqref{eq:R22}, and more in general in terms of the geometry associated with the control system \eqref{eq:contprob}.
\subsection{Canonical volume form} 
Owing to Proposition~\ref{prop:cf}, we complete the canonical frame along the extremal $t\mapsto \lambda(t)$, $t\in[0,T]$. Referring to the same kind of calculations as in the previous section, we find \be\label{eq:F_1}
F_1(t)=R_{22}(t)E_2(t)-\dot{F}_2(t)=R_{22}(t)\dot{\partial}_{h_2}(t)+\dddot{\partial}_{h_2}(t)\in T^*_{\lambda(t)}TM.
\ee
Then we declare $\theta_1(t)$, $\theta_2(t)$ to be the coframe dual to $Z_i(t)=\pi_*(F_i)(t)\in T_{x(t)}M$, $i=1,2$, and set
$\xi_{x(t)}:=\theta_1(t)\wedge \theta_2(t)$. Intuitively, $\xi_{x(\cdot)}$ is the volume form canonically associated with \eqref{eq:contprob}, along the geodesic $t\mapsto x(t)$. By explicit computations, we find that $Z_1=X_2-BX_1$, and $Z_2=-X_1$. 
In local coordinates, it is immediate to establish that 
\be\label{eq:xi}\xi(x_1,x_2)=-\frac{1}{\partial_{x_1}\alpha_2(x_1,x_2)}dx_1\wedge d x_2,
\ee 
so that $\xi$ coincides indeed with the volume form $\mu$ chosen in \eqref{eq:w} and described in \eqref{eq:mucoord}. Alternatively, we could have directly argued that by definition, $\theta_1(X_2)(t)=-\theta_2(X_1)(t)=1$, implying $\xi_{x(t)}(X_1, [X_0,X_1])=1$ and  $\xi_{x(t)}=\mu_{x(t)}$.

Since by assumption $S=\partial_{x_1}\alpha_2(0)\neq 0$, it is not restrictive to assume that $\mu$ does not change sign on the whole coordinate chart.

\subsection{Proof of Theorem \ref{t:main}} 
We begin by observing that 
\be
\label{eq:lap}\partial_t-X_0-\frac12 \left(X_1^2+\mathrm{div}_\mu(X_1)X_1\right),
\ee 
can be interpreted in the form $\partial_t-X_0-\Delta^\mu_{H}$, where $\Delta^\mu_{H}$ denotes the horizontal Laplacian \cite{hypoelliptic,boscain2013laplace} (also called sub-Laplacian) naturally associated with \eqref{eq:contprob} and the canonical volume form $\mu$ characterized in \eqref{eq:xi}.

On the one hand, by the general theory developed for the expansion in \eqref{eq:duhamel} and the 
result of Theorem~\ref{thm:fc}, with $\sqrt{t}$ in place of $\varepsilon$, we know that the asymptotic development of $p(t,0,0)$ takes the form
\begin{align}\label{eq:fundcoor}\\
	p(t,0,0)=\frac{1}{2\pi t^2\sqrt{\det(G_1)}}&\left(1+t\left(-\frac{1}{2}\alpha_1^2(0)-\frac{1}{2}\bigg(\partial_{x_1}\alpha_1(0)+\partial_{x_2}\alpha_2(0)\right.\right.\\
	&-\alpha_1(0)\frac{\partial_{x_1}^2\alpha_2(0)}{\partial_{x_1}\alpha_2(0)}\bigg)\left.\left.-\frac{12}{35}\left(\frac{\partial_{x_1}^2\alpha_2(0)}{\partial_{x_1}\alpha_2(0)}\right)^2+\frac{3}{14}\frac{\partial_{x_1}^3\alpha_2(0)}{\partial_{x_1}\alpha_2(0)}\right)+o(t)\right)
\end{align}
within a coordinate chart $U$ centered at $x_0$. 

On the other hand, reading in coordinates $(x_1,x_2)$ the structural equations (compare also with \eqref{eq:gradedcoord})
\be\label{eq:expansion}
	X_{2}=[X_{0},X_{1}],\quad[X_{1},X_{2}]=c_{12}^{1}X_{1}+c_{12}^{2}X_{2},\quad
	[X_{0},X_{2}]=c_{02}^{1}X_{1}+c_{02}^{2}X_{2},
\ee
we see that, in particular,
\be
	c_{12}^2(x_1,x_2)=\frac{\partial_{x_1}^2\alpha_2(x_1,x_2)}{\partial_{x_1}\alpha_2(x_1,x_2)}\quad\mathrm{ and }\quad X_1(c_{12}^2)(x_1,x_2)=\frac{\partial_{x_1}^3\alpha_2(x_1,x_2)}{\partial_{x_1}\alpha_2(x_1,x_2)}-\left(\frac{\partial_{x_1}^2\alpha_2(x_1,x_2)}{\partial_{x_1}\alpha_2(x_1,x_2)}\right)^2.
\ee

It is now easy to reinterpret \eqref{eq:fundcoor} in an intrinsic way. More precisely, we have shown that, under the hypotheses of Theorem~\ref{t:main}, 
the fundamental solution $p$ to \eqref{eq:lap} 
admits the following small-time asymptotic expansion on the diagonal around $x_0$, namely:
\begin{align}
	p(t,x_0,x_0)&=\frac{\sqrt{12}}{2\pi t^2
	}\bigg(1+t\bigg(-\frac12 |X_0(x_0)|^2-\frac12 \mathrm{div}_\mu(X_0)(x_0)+\frac1{14}K_1(x_0)-\frac{1}{70}K_2(x_0)^2\bigg)+o(t) \bigg),
\end{align}
where 
$K_1$ and $K_2$ are defined as in Definition \ref{d:kk}. This ends the proof of Theorem \ref{t:main}.

\subsection*{Acknowledgements}  We are grateful to Emmanuel Tr\'elat and Yves Colin de Verdi\`ere for useful advices concerning some technical points presented in Section~\ref{subsec:conv} and in Proposition~\ref{prop:convoncomp}. We also warmly thank Dario Prandi for helpful discussions, and  Elisa Paoli who gave a first input to this project. 

This research has been supported by the ANR project SRGI Sub-Riemannian Geometry and Interactions", contract number ANR-15-CE40-0018.

	\bibliographystyle{abbrv}
	\bibliography{Biblio}		
\end{document}